\documentclass[12pt,twoside]{amsart}
\usepackage{latexsym,amsfonts,amsmath,amssymb,color}

\def\Im{\mathop{\rm Im}\nolimits}
\def\Re{\mathop{\rm Re}\nolimits}
\def\supp{\mathop{\rm supp}\nolimits}
\def\ds{\displaystyle}

\def\op{\mathop{\rm op}\nolimits}
\def\R{\mathbb R}

\def\N{\mathbb N}
\def\3|{\|\hspace*{-1pt}|}
\newcommand{\Osc}{\mathcal{O}}
\newcommand{\B}{\mathcal{B}}

\newcommand{\afrac}[2]{\genfrac{}{}{0pt}{1}{#1}{#2}}

\newcommand{\beqsn}{\arraycolsep1.5pt\begin{eqnarray*}}
\newcommand{\eeqsn}{\end{eqnarray*}\arraycolsep5pt}
\newcommand{\beqs}{\arraycolsep1.5pt\begin{eqnarray}}
\newcommand{\eeqs}{\end{eqnarray}\arraycolsep5pt}
\newcommand*{\dbar}{\mbox{\dj}}	

\newtheorem{Th}{Theorem}[section]
\newtheorem{Rem}[Th]{Remark}

\newtheorem{Lemma}[Th]{Lemma}

\newtheorem{Prop}[Th]{Proposition}

\catcode`\@=11

\renewcommand{\section}%
   {\setcounter{equation}{0}\@startsection {section}{1}{\z@}{-3.5ex plus -1ex
  minus -.2ex}{2.3ex plus .2ex}{\Large\bf}}

\title{A Necessary condition for $H^\infty $ well-posedness of 
$p$-evolution equations}
\author[Ascanelli]{A.~Ascanelli}
\address{Alessia Ascanelli\\
Dipartimento di Matematica e Informatica\\Universit\`a di Ferrara\\
Via Machiavelli n.~30\\
44121 Ferrara\\
Italy}
\email{alessia.ascanelli@unife.it}
\author[Boiti]{C.~Boiti}
\address{Chiara Boiti\\
Dipartimento di Matematica e Informatica \\Universit\`a di Ferrara\\
Via Ma\-chia\-vel\-li n.~30\\
44121 Ferrara\\
Italy}
\email{chiara.boiti@unife.it}
\author[Zanghirati]{L.~Zanghirati}
\address{Luisa Zanghirati\\
Dipartimento di Matematica e Informatica \\Universit\`a di Ferrara\\
Via Ma\-chia\-vel\-li n.~30\\
44121 Ferrara\\
Italy}
\email{zan@unife.it}
\begin{document}


\begin{abstract}
We consider $p$-evolution equations, for $p\geq2$, with complex valued
coefficients. We prove that a necessary condition for $H^\infty$ well-posedness
of
the associated Cauchy problem is that
the imaginary part of the coefficient of the subprincipal part
(in the sense of Petrowski) satisfies a decay estimate as $|x|\to+\infty$.
\end{abstract}

\maketitle
\markboth{\sc A necessary condition for
$H^\infty$ well-posedness \ldots}{\sc A.~Ascanelli, C.~Boiti, L.~Zanghirati}

\section{Introduction and main result}

Given an integer $p\geq 2$, we consider in $[0,T]\times\R$ the 
linear partial differential operator $P$ of the form
\begin{equation}\label{op}
P(t,x,D_t,D_x)=D_t+a_p(t)D_x^p+\sum_{j=0}^{p-1}
a_j(t,x)D_x^j\,,
\end{equation}
with $D=\frac 1i \partial$, $a_p\in C([0,T];\R)$ and
$a_j\in C([0,T];{\mathcal B}^\infty)$ for $0\leq j\leq p-1$, (here
${\mathcal B}^\infty={\mathcal B}^\infty(\R_x)$ is the space of complex 
valued functions which are bounded on $\R_x$ together with all their
$x$-derivatives). We are dealing with a non-kowalewskian evolution operator; 
anisotropic evolution operators of the form (\ref{op}) are usually called 
$p-$evolution operators. The condition that $a_p$ is real valued means that 
the principal symbol 
(in the sense of Petrowski)
of $P$ has the real characteristic $\tau=-a_p(t)\xi^p$; by the Lax-Mizohata 
theorem (cf. \cite{M}), this is a necessary condition to have a unique 
solution, in Sobolev spaces, of the Cauchy problem 
\beqs
\label{CP1}
\begin{cases}
P(t,x,D_t,D_x)u(t,x)=f(t,x) & (t,x)\in[0,T]\times\R\cr
u(0,x)=g(x) & x\in\R,
\end{cases}
\eeqs
in a neighborhood of $t=0$.
We notice that for
$p=2$ the operator is of Schr\"odinger type, for $p=3$ we have the 
same principal part as the Korteweg-De Vries equation.
Many 
results of well-posedness 
in Sobolev spaces of \eqref{CP1}
are available under the assumption that all the coefficients 
$a_j$ of \eqref{op} are real (see, for instance,
\cite{Ag1}, \cite{Ag2}, \cite{AZ}, 
\cite{AC}, \cite{CC2}, \cite{CHR}). On the contrary, when the coefficients 
$a_j(t,x)$ for $1\leq j\leq p-1$ 
are not real, the theory is well developed only in the case $p=2$: we 
know from the pioneering papers 
\cite{I1}, \cite{I2} that a decay condition as $|x|\to+\infty$ on
$\Im a_1$ is necessary and sufficient for well-posedness of the 
Cauchy problem \eqref{CP1} in $H^\infty$.
A necessary condition is given also in \cite{CR3} for $p=2$.
Sufficient conditions for well-posedness in $H^\infty$ and/or Gevrey classes
for 2 or 
$3-$evolution equations have been given by many authors (see, for instance, 
\cite{D}, \cite{T}, \cite{B}, \cite{KB}, \cite{ACC}, \cite{CR1}, 
\cite{CR2}, \cite{CC1}).
The general case $p\geq2$ has been recently considered in \cite{ABZ1}, proving
$H^\infty$ well-posedness of the Cauchy problem \eqref{CP1} under suitable
decay conditions, as $|x|\to+\infty$, on $\Im D_x^\beta a_j$, for
$j\leq p-1$ and $[\beta/2]\leq j-1$.
These results have been extended to the case of weighted Sobolev spaces in
\cite{ACpp}, to the case of first order systems of pseudo-differential
operators in \cite{AB1}, to the case of higher order equations
in \cite{AB2}, and to semi-linear
equations in \cite{ABKrakow}, \cite{ABZ2} for some model equations and in
\cite{ABsemi} for the general case.

As far as we know, there are no results available about necessary 
conditions for $H^\infty$ well-posedness for $p$-evolution equations, 
$p\geq 3$. 

In this paper we give a necessary condition 
for well-posedness of the
Cauchy problem \eqref{CP1} in $H^\infty$, generalizing to the case
$p\geq2$ the necessary condition given by Ichinose in \cite{I1} for
$p=2$.
More precisely, in \cite{I1} Ichinose considered, for $x\in\R^n$,
the operator 
\begin{equation}\label{pichi}
P=D_t-a_2\triangle_x+\sum_{j=1}^na_1^{(j)}(x)D_{x_j}+c(x),
\end{equation}
with $a_2\in(0,1]$ and $a_1^{(j)},c\in{\mathcal B}^\infty(\R^n)$.
He proved that a necessary condition for $H^\infty$ well-posedness
of the associated Cauchy problem
is the existence of non-negative constants $M,\ N$ such that
\begin{equation}
\label{ichi}
\sup_{x\in\R^n\!,\,\omega\in S^{n-1}}\left\vert\sum_{j=1}^n\int_0^\varrho 
\Im a_1^{(j)}(x+2a_2\theta\omega)\omega_jd\theta
\right\vert\leq M\log(1+\varrho)+N
\qquad\forall\varrho>0,
\end{equation}
where $S^{n-1}$ is the unit sphere in $\R^{n}$.
The same condition is also sufficient (cf. \cite{I2}) only in the 
case of space dimension $n=1$. 

In this paper we assume
that there exists a constant $m>0$ such that
\begin{equation}\label{a3}
|a_p(t)|\geq m\qquad \forall t\in [0,T]
\end{equation}
and prove the following:
\begin{Th}
\label{th1}
Let $P$ be the operator in \eqref{op} with $a_p\in C([0,T];\R)$
satisfying \eqref{a3} and $a_j\in C([0,T];\B^\infty)$ for $0\leq j\leq p-1$. 
A necessary 
condition for the Cauchy problem \eqref{CP1} to be well-posed in $H^\infty$ is 
the existence of constants $M,N>0$ such that:
\begin{equation}
\label{CN2}
\sup_{x\in\R}\min_{0\leq\tau\leq t\leq T}\int_{-\varrho}^\varrho 
\Im a_{p-1}(t,x+p a_p(\tau)\theta)d\theta\leq M\log(1+\varrho)+N,\qquad 
\forall \varrho>0.
\end{equation}
\end{Th}

\begin{Rem}
\begin{em}
If the coefficient $a_p(t)$ vanishes at some point of the interval $[0,T]$, 
the well-posedness in $H^\infty$ of the Cauchy problem \eqref{CP1} 
may fail to be true also if the necessary condition 
\eqref{CN2} is satisfied (see \cite{CR3}, \cite{ABZ1}).
In \cite{ABZ1}, for instance, 
a Levi-type condition of the form
\beqs
\label{levi}
|\Im a_{p-1}(t,x)|\leq Ca_p(t)/\langle x\rangle
\eeqs
is needed to 
weaken \eqref{a3} into 
$a_p(t)\geq0$ for proving $H^\infty$ well-posedness of \eqref{CP1}.
Notice that condition \eqref{levi} is consistent with \eqref{CN2}.
In the present paper we focus 
on the fact 
that the dependence on space of the coefficient of the subprincipal part is 
allowed only accompanied by a decay condition at infinity.
\end{em}
\end{Rem}


\section{Idea of the proof and auxiliary tools.}
\label{sec2}
We prove Theorem \ref{th1} by contradiction, taking $f\equiv 0$ 
without any loss of generality.

We assume the Cauchy problem \eqref{CP1} to be well-posed, so that
for every $g\in H^\infty$, there exists a unique 
$u\in C([0,T];H^\infty)$
solution of \eqref{CP1} and there exists $q\in\N_0:=\N\cup\{0\}$ 
and $C>0$ such that
\begin{equation}\label{q}
\|u(t,\cdot)\|_0\leq C\|g\|_q\qquad\forall t\in [0,T],
\end{equation}
where $\|\cdot\|_s$ stands for the norm in the Sobolev space $H^s$
(we shall write $\|\cdot\|:=\|\cdot\|_0$ for simplicity).

Then we assume, by contradiction, that \eqref{CN2} does not hold.
This implies that, for every $M>0$ and $k\in\N$ there exist a sequence
of points $x_k\in\R$ and a sequence $\varrho_k\to+\infty$ such that
\beqs
\label{-roro}
\int_{-\varrho_k}^{\varrho_k}\Im a_{p-1}(t,x_k+pa_p(\tau)\theta)d\theta\geq
M\log(1+\varrho_k)+k\qquad\forall 0\leq\tau\leq t\leq T.
\eeqs
We can then construct a sequence of initial data $g_k$ localized at high 
frequency $n_k:=\varrho_k^a$,
for suitable $a>0$, so
obtaining a sequence $u_k$ of solutions of the corresponding Cauchy 
problem. Further localizing these solutions 
in the phase space along the trajectory of 
the hamiltonian $a_p(t)\xi^p$, we produce a sequence of functions 
$v_k^{\alpha,\beta}$ (for $\alpha,\beta\in\N_0$) satisfying some energy 
estimates, because of \eqref{q}.

Taking, finally, a suitable linear combination $\sigma_k(t)$ of the
$L^2$-norms $\|v_k^{\alpha,\beta}(t,\cdot)\|$, we obtain, in Section~\ref{sec3},
that \eqref{-roro} implies an estimate from below of $\sigma_k(t)$; 
this estimate will contradict an estimate from above 
for $\sigma_k(t)$ which is stated and proved in Section~\ref{sec4}.

In this section we discuss condition \eqref{CN2}, construct the
sequence $\{v_k^{\alpha,\beta}\}$ and collect some estimates that will be crucial
in the proofs of the contradictory estimates from below 
and from above of $\sigma_k(t)$.

The next section is completely devoted to the proof of the estimate from 
below \eqref{formulabelow}.

In Section~\ref{sec4} we give the estimate from above 
\eqref{formulaabove}, and finally prove Theorem \ref{th1}.

Let us start by remarking that if condition \eqref{CN2} does not hold, then
at least one of the following two conditions does not hold:
\beqs
\label{A}
\sup_{x\in\R}\min_{0\leq \tau\leq t\leq T}\int_0^\varrho 
\Im a_{p-1}(t,x+p a_p(\tau)\theta)d\theta\leq M\log(1+\varrho)+N,\qquad 
\forall \varrho>0,
\eeqs
or
\beqs
\label{B}
\sup_{x\in\R}\min_{0\leq \tau\leq t\leq T}\int_{-\varrho}^0
\Im a_{p-1}(t,x+p a_p(\tau)\theta)d\theta\leq M\log(1+\varrho)+N,\qquad 
\forall \varrho>0.
\eeqs
Since
\beqsn
\int_{-\varrho}^0
\Im a_{p-1}(t,x+p a_p(\tau)\theta)d\theta=
\int_0^\varrho\Im a_{p-1}(t,x-p a_p(\tau)\theta)d\theta,
\eeqsn
we can assume, without any loss of generality, 
that \eqref{A} does not hold and obtain then a contradiction 
(if \eqref{B} does not hold we argue in the same way taking $-a_p$ instead
of $a_p$).

The following lemma will be the key to obtain the desired estimate from below
\eqref{formulabelow}:
\begin{Lemma}
\label{lemma1}
If \eqref{A} does not hold, then for every $M>0$ and $k\in\N$ there
esist $x_k\in\R$ and $\varrho_k>0$ such that:
\begin{itemize}
\item[(i)]
$\varrho_k\to+\infty$;
\item[(ii)]
$\ds\int_0^{\varrho_k}\Im a_{p-1}(t,x_k+p a_p(\tau)\theta)d\theta\geq 
M\log(1+\varrho_k)+k\qquad\forall 0\leq\tau\leq t\leq T$;
\item[(iii)]
$\ds\int_0^\varrho\Im a_{p-1}(t,x_k+p a_p(\tau)\theta)d\theta\geq0\qquad\forall 
\varrho\in[0,\varrho_k],\ 0\leq\tau\leq t\leq T$.
\end{itemize}
\end{Lemma}
\begin{proof}
If \eqref{A} fails to be true, then for every $M>0$, 
$k\in\N$  there exist
$y_k\in\R$ and $\delta_k>0$ such that for all $0\leq\tau\leq t\leq T$
\beqs
\label{deltak}
\int_0^{\delta_k}\Im a_{p-1}(t,y_k+p\theta a_p(\tau))
d\theta\geq M\log(1+\delta_k)+k.
\eeqs
Let us set, for $s\in[0,\delta_k]$,
\beqsn
F_k(s):=\int_0^s\Im a_{p-1}(t,y_k+p\theta a_p(\tau))d\theta
\eeqsn
and let $s_k$ be point of minimum of $F_k$ on $[0,\delta_k]$. Define then
\beqsn
x_k:=&&y_k+ps_ka_p(\tau)\\
\varrho_k:=&&\delta_k-s_k.
\eeqsn
Remark that all $y_k,\delta_k,s_k,x_k,\varrho_k$ depend also on $M$.

For all $s\in[0,\varrho_k]\subseteq[0,\delta_k]$:
\beqs
\nonumber
\int_0^s\Im a_{p-1}(t,x_k+p a_p(\tau)\theta)d\theta
=&&\int_{s_k}^{s+s_k}\Im a_{p-1}(t,y_k+p a_p(\tau)\theta')d\theta'\\
\label{2star}
=&&F_k(s+s_k)-F_k(s_k)\geq0
\eeqs
by definition of $s_k$. This proves
$(iii)$.

Moreover, $F_k(s_k)\leq F_k(0)=0$ and hence, from \eqref{2star} 
and \eqref{deltak}:
\beqsn
\int_0^{\varrho_k}\Im a_{p-1}(t,x_k+p a_p(\tau)\theta)d\theta=&&
\int_0^{\delta_k-s_k}\Im a_{p-1}(t,x_k+p a_p(\tau)\theta)d\theta\\
=&&F_k(\delta_k)-F_k(s_k)\geq F_k(\delta_k)\\
\geq&&M\log(1+\delta_k)+k\geq M\log(1+\varrho_k)+k,
\eeqsn
proving $(ii)$.

Finally, the last inequality implies, for $k\to+\infty$,
\beqsn
\int_0^{\varrho_k}\Im a_{p-1}(t,x_k+p a_p(\tau)\theta)d\theta\geq k\to+\infty
\eeqsn
and hence $\varrho_k\to+\infty$, because $a_{p-1}\in\B^\infty$.
\end{proof}

\subsection{Solutions with high frequency initial data}
Let us fix, here and throughout all the paper, a cut-off 
function $h\in C^\infty(\R)$, such that 
\begin{equation}\label{h}
h(y)=\left\{\begin{array}{ll}
1 & |y|\leq 1/4
\\
0 & |y|\geq 1/2,
\end{array}\right.
\end{equation}
and a rapidly decreasing function $\psi$ such that $\psi(0)=2$ and 
$$\supp \hat\psi\subseteq\{\xi\in\R:\  h(\xi)=1\}.$$
Define then
\begin{equation}\label{gk}
g_k(x)=e^{i(x-x_k)n}\psi(x-x_k), 
\end{equation}
where 
\begin{equation}\label{n}
n:=\varrho_k^{a}
\end{equation}
for some $a>0$ to be chosen later on (see \eqref{sceltamu}), 
and $x_k, \varrho_k$ as in Lemma \ref{lemma1}.
Note that
\beqs
\label{fgk}
\hat{g}_k(\xi)=e^{-ix_k\xi}\hat{\psi}(\xi-n),
\eeqs
so ${g}_k$ is localized in the phase space around the point $(x_k,n)$.

Denote by $u_k\in C([0,T]; H^\infty)$
the solution of the Cauchy problem  
\beqs
\label{CP2}
\begin{cases}
P(t,x,D_t,D_x)u_k(t,x)=0 & (t,x)\in[0,T]\times\R\cr
u_k(0,x)=g_k(x) & x\in\R.
\end{cases}
\eeqs
Then, by \eqref{q} and \eqref{fgk} we have, for all $t\in [0,T]$:
\beqs\nonumber
 \|u_k(t,\cdot)\|\leq C\|g_k\|_q
&=&C(2\pi)^{-1/2}\left(\ds\int_{\R}
\langle\xi\rangle^{2q}|\hat g_k(\xi)|^2d\xi\right)^{1/2}
 \\
  \nonumber
 &\leq&C_q (2\pi)^{-1/2}\langle n\rangle^{q}\left(\ds\int_{\R}
\langle\theta\rangle^{2q}|\hat \psi(\theta)|^2d\theta\right)^{1/2}
\\
\label{C'T}
&\leq& C'_qn^q,
\eeqs
for some $C_q,C'_q>0$.

\subsection{A localizing operator}
In this subsection we define, by giving its symbol $w_{n,k}(t,x,\xi)$, 
a pseudo-differential operator $W_{n,k}(t,x,D_x)$
which localizes the 
solutions of \eqref{CP1} in the phase space along the trajectory of 
the hamiltonian $a_p(t)\xi^p$.

Let $w_{n,k}(t,x,\xi)$ be the solution of the Hamilton's equation of motion
\begin{equation}
\label{eqw}
\begin{cases}
\partial_tw_{n,k}=\left\{w_{n,k},-a_p(t)\xi^p\right\}
\\
w_{n,k}(0,x,\xi)=w_{0,n,k}(x,\xi):=\varrho_k^{1/2}h
\left(\varrho_k(x-x_k)\right)h\left(\varrho_k^\mu(\xi/n-1)\right),
\end{cases}
\end{equation}
with $\mu>0$ to be chosen later (see \eqref{sceltamu}),
where $\{\cdot,\cdot\}$ denotes the Poisson brackets defined by
\[\{p(x,\xi),q(x,\xi)\}=\partial_xp(x,\xi)\partial_\xi 
q(x,\xi)-\partial_\xi p(x,\xi)\partial_xq(x,\xi).\]
Computing the Poisson brackets, equation \eqref{eqw} reduces to
\begin{equation}
\label{215}
\begin{cases}
\left(\partial_t +pa_p(t)\xi^{p-1}\partial_x\right) w_{n,k}=0
\\
w_{n,k}(0,x,\xi)=w_{0,n,k}(x,\xi),
\end{cases}
\end{equation}
which admits the solution
\[
w_{n,k}(t,x,\xi)=w_{0,n,k}(x-pA_p(t)\xi^{p-1},\xi),\quad A_p(t)=
\int_0^t a_p(\tau)d\tau.
\]
We thus obtain
\begin{equation}
\label{defwnk}
w_{n,k}(t,x,\xi):=\varrho_k^{1/2}h
\left(\varrho_k(x-x_k-pA_p(t)\xi^{p-1})\right)
h\left(\varrho_k^\mu(\xi/n-1)\right).
\end{equation}
The following lemma shows that the symbol $w_{n,k}(t,x,\xi)$ is 
supported in a neighborhood of
the solution $(x_k+pA_p(t)n^{p-1},n)$ of 
the Hamilton's canonical equation with 
initial data $(x_k,n)$; moreover it introduces the sequence of symbols 
$w_{n,k}^{\alpha,\beta}$ which naturally appear in the computation of
$\partial_\xi^\alpha D_x^\beta w_{n,k}$.

\begin{Lemma}
\label{lemma2}
Let us define, for $\alpha,\beta\in\N_0$,
$\mu\geq2$ and $n$ as in \eqref{n},
the symbols
\begin{equation}
\label{wnkab}
w_{n,k}^{\alpha,\beta}(t,x,\xi):=\left.\varrho_k^{1/2}(\partial_x^\alpha h)(x)
(\partial_\xi^\beta h)(\xi)
\right\vert_{\afrac{x=\varrho_k(x-x_k-pA_p(t)\xi^{p-1})}{\hspace*{-12mm}\xi
=\varrho_k^\mu(\xi/n-1)}}.
\end{equation}
Then, for every 
$t\in\left[0,\ds\frac{\varrho_k}{n^{p-1}}\right]$ we have that
\beqsn
\supp w_{n,k}^{\alpha,\beta}(t)\subseteq\left\{(x,\xi):\ 
\vert x-(x_k+pA_p(t)n^{p-1})\vert\leq \frac{c_p}{\varrho_k},\ 
\vert\xi/n-1\vert\leq \frac{1}{2\varrho_k^\mu}\right\},
\eeqsn
for $c_p=\max\{1,p2^{p-1}\sup_{[0,T]}|a_p|\}$, if $k$ is large enough.
\end{Lemma}
\begin{proof}
The estimate $\vert\xi/n-1\vert\leq 1/(2\varrho_k^\mu)$ trivially 
follows by definition \eqref{wnkab} and by \eqref{h}. 
Moreover, \eqref{h} implies, for $t\in[0,\varrho_k/n^{p-1}]$, $\mu\geq2$
and $(x,\xi)\in\supp w_{n,k}^{\alpha,\beta}$:
\beqsn
\vert x-(x_k+pA_p(t)n^{p-1})\vert&\leq&\vert x-(x_k+pA_p(t)\xi^{p-1})
\vert+p|A_p(t)|n^{p-1}\left\vert \left(\frac{\xi}{n}\right)^{p-1}-1\right\vert
\\
&\leq&\frac 1{2\varrho_k}+p\sup_{[0,T]}|a_p| t n^{p-1}
\cdot\left|\frac\xi n-1\right|\cdot\left|\left(\frac{\xi}{n}\right)^{p-2}
+\left(\frac{\xi}{n}\right)^{p-3}+\ldots+1\right|\\
&\leq&\frac 1{2\varrho_k}+p\sup_{[0,T]}|a_p| \frac {\varrho_k}{2\varrho_k^\mu}
\left[2^{p-2}
+2^{p-3}+\ldots+1\right]\\
&\leq&\frac 1{2\varrho_k}+p\sup_{[0,T]}|a_p| \frac 1{2\varrho_k^{\mu-1}}2^{p-1}
\leq \frac{c_p}{\varrho_k},
\eeqsn
for $c_p=\max\{1,p2^{p-1}\sup_{[0,T]}|a_p|\}$,
since $\xi/n\leq|\xi/n-1|+1\leq1/(2\varrho_k^\mu)+1\leq2$ for $k$ 
large enough.
\end{proof}

As a consequence of Lemma \ref{lemma2}, 
we localize, in the phase space, the solution of \eqref{CP2}, 
defining 
\begin{equation}\label{vkab}
v_k^{\alpha,\beta}(t,x):=W_{n,k}^{\alpha,\beta}(t,x,D_x)u_k(t,x),
\end{equation}
where $W_{n,k}^{\alpha,\beta}(t,x,D_x)$ is the pseudo-differential 
operator with symbol $w_{n,k}^{\alpha,\beta}(t,x,\xi)$.
We shall denote, throughout all the paper, $W_{n,k}:=W_{n,k}^{0,0}(t,x,D_x)$ 
and  $v_k:=v_k^{0,0}$ for simplicity.
\subsection{Useful estimates}
In the next sections we need
estimates of the $L^2$-norms of the functions 
$v_k^{\alpha,\beta}$ and of both operators $W_{n,k}^{\alpha,\beta}$ 
and $[a_j,W_{n,k}]$ acting on $u_k$ . In this subsection we state and 
prove all these estimates. Proofs are quite technical, and the main 
tools for obtaining them, collected in
Appendix A, are the Calder\'on-Vaillancourt's Theorem \ref{thC}
and a skillful use of the expansion formula of the symbol of the 
product of two pseudo-differential operators (Theorems \ref{thA} and 
\ref{thB}). 
To avoid losing his train of thought, the reader can skip 
these estimates at a first reading,
passing directly to Section~\ref{sec3} and coming back to the estimates at 
the moment of their application. 

To estimate the $L^2$-norms of $v_k$ and of 
$v_k^{\alpha,\beta}$ we first need estimates of the semi-norms 
$|\cdot|^{0}_{\ell,\ell}$ of the
symbols $w_{n,k}^{\alpha,\beta}\in S^0_{0,0}$, defined in formula \eqref{semi000}
of the Appendix.
\begin{Lemma}
\label{lemmaAB1}
Let $n=\varrho_k^a$ with $a\geq \mu\geq 2$, and 
$t\in \left[0,\ds\frac{\varrho_k}{n^{p-1}}\right]$.
Then, for every $\alpha,\beta\in\N_0$ we have, for $k$ large enough:
\begin{itemize}
\item[(i)]
for every $\gamma,\sigma\in\N_0$ there exists a constant
$C_{\alpha,\beta,\gamma,\sigma}>0$ such that, for all 
$(t,x,\xi)\in[0,\frac{\varrho_k}{n^{p-1}}]\times\R^2$:
\beqsn
|\partial_\xi^\gamma\partial_x^\sigma w_{n,k}^{\alpha,\beta}(t,x,\xi)|
\leq C_{\alpha,\beta,\gamma,\sigma}\varrho_k^{\frac12+\sigma}
\left(\frac{\varrho_k^\mu}{n}\right)^\gamma;
\eeqsn
\item[(ii)]
for every $\ell\in\N$ there exists $C_{\alpha,\beta,\ell}>0$ such that
\beqsn
\left\vert w_{n,k}^{\alpha,\beta}\right\vert_{\ell,\ell}^{0}\leq 
C_{\alpha,\beta,\ell}\varrho_k^{\frac 12+\ell};
\eeqsn
\item[(iii)]
for every $h\in \N_0$ and $\nu,\ell\in\N$ there exists 
$C_{\alpha,\beta,\nu,\ell}>0$ such that
\beqs\label{AB5}
|\xi^h\partial_\xi^\nu w_{n,k}^{\alpha,\beta}|_{\ell,\ell}^0
\leq C_{\alpha,\beta,\nu,\ell}\,n^h\varrho_k^{\frac12+\ell}
\left(\frac{\varrho_k^\mu}{n}\right)^\nu.
\eeqs
\end{itemize}
\end{Lemma}

\begin{proof}
Let us write
\beqs
\nonumber
\partial_\xi^\gamma\partial_x^\sigma w_{n,k}^{\alpha,\beta}(t,x,\xi)=&&
\varrho_k^\sigma\partial_\xi^\gamma w_{n,k}^{\alpha+\sigma,\beta}(t,x,\xi)\\
\label{CC1}
=&&
\varrho_k^{\sigma+\frac12}\sum_{\gamma_1+\gamma_2=\gamma}C_\gamma
\partial_\xi^{\gamma_1}h^{(\alpha+\sigma)}(\varrho_k(x-x_k-pA_p(t)\xi^{p-1}))\cdot
\partial_\xi^{\gamma_2}h^{(\beta)}(\varrho_k^\mu(\xi/n-1)).
\eeqs
Since $|\xi|\leq 2n$ on $\supp w_{n,k}^{\alpha,\beta}$, by Lemma \ref{lemma2} we 
have that
\beqsn
\left\vert \partial_\xi^{\gamma_1}h^{(\alpha+\sigma)}
\left(\varrho_k(x-x_k-pA_p(t)\xi^{p-1})\right)\right\vert&\leq& 
C_{\alpha,\sigma,\gamma_1}(A_p(t)|\xi|^{p-2}\varrho_k)^{\gamma_1}\\
&\leq& C_{\alpha,\sigma,\gamma_1}\left(\sup_{[0,T]}|a_p|\cdot t
\cdot|\xi|^{p-2}\varrho_k \right)^{\gamma_1}
\leq C'_{\alpha,\sigma,\gamma_1}\left(\frac{\varrho_k^2}{n}\right)^{\gamma_1}
\eeqsn
for $t\in[0,\varrho_k/n^{p-1}]$.
Moreover,
\beqsn
\left\vert \partial_\xi^{\gamma_2}h^{(\beta)}
\left(\varrho_k^\mu(\xi/n-1)\right)\right\vert&\leq& C_{\gamma_2,\beta}
\left(\frac{\varrho_k^\mu}{n}\right)^{\gamma_2}.
\eeqsn
Substituting in \eqref{CC1} we thus obtain $(i)$, since $\mu\geq2$.

From $(i)$ and $a\geq\mu$ we get
\beqsn
\left\vert w_{n,k}^{\alpha,\beta}\right\vert_{\ell,\ell}^{0}
=\sup_{\afrac{\gamma,\sigma\leq \ell}{x,\xi\in\R}}\left\vert 
\partial_\xi^\gamma\partial_x^\sigma w_{n,k}^{\alpha,\beta}(t,x,\xi)\right\vert
\leq 
\sup_{\afrac{\gamma,\sigma\leq \ell}{x,\xi\in\R}}C_{\alpha,\beta,\gamma,\sigma}\,
\varrho_k^{\frac 12+\sigma}\left(\frac{\varrho_k^\mu}{n}\right)^{\gamma}
\leq C_{\alpha,\beta,\ell}\,
\varrho_k^{\frac 12+\ell}
\eeqsn
i.e. also $(ii)$ is satisfied.

Finally, $(iii)$ follows from $(i)$, since $|\xi|\leq2n$ on $\supp
w_{n,k}^{\alpha,\beta}$ and $a\geq\mu$.
\end{proof}

We are now ready to estimate $\|v_k\|$. By Calder\'on-Vaillancourt's 
Theorem \ref{thC}, $(ii)$ of Lemma \ref{lemmaAB1} 
and \eqref{C'T}, we have that for all $t\in[0,\varrho_k/n^{p-1}]$
\beqs\label{nomme}
\|v_k(t,\cdot)\|&=&\|W_{n,k}^{0,0}(t,\cdot,D_x)u_k(t,\cdot)\|
\leq C\left\vert w_{n,k}(t,x,\xi)\right\vert_{2,2}^{0}\|u_k(t,\cdot)\|
\leq C'\varrho_k^{\frac 12+2}n^q
\eeqs
for some $C,C'>0$; similarly, for every $\alpha,\beta\in\N_0$, it follows that
\beqs\label{normevkab}
\|v_k^{\alpha,\beta}(t,\cdot)\|\leq C_{\alpha,\beta}\varrho_k^{\frac 12+2}n^q
=C_{\alpha,\beta}\varrho_k^{\frac 12+2+aq}\qquad\forall t\in[0,\varrho_k/n^{p-1}]
\eeqs
for some $C_{\alpha,\beta}>0$.
To estimate also the derivatives of $v_k^{\alpha,\beta}$ we need the following:

\begin{Lemma}
\label{lemmaAB3}
Let $n=\varrho_k^a$  with $a\geq\mu$.
For every $\nu,r\in\N$ and $\alpha,\beta\in\N_0$ there exists
$C_{\alpha,\beta,r,\nu}>0$ such that for all $t\in[0,\varrho_k/n^{p-1}]$ 
with $k$ large enough:
\beqsn
\|D_x^rv_k^{\alpha,\beta}(t,\cdot)\|\leq c_1n^r\|v_k^{\alpha,\beta}\|
+C_{\alpha,\beta,r,\nu}\varrho_k^{4+\frac12}
\left(\frac{\varrho_k^{\mu+1}}{n}\right)^\nu
n^{r+q},
\eeqsn
for a fixed constant $c_1>0$.\end{Lemma}
\begin{proof}
We define the function
\beqs\label{chi1k}
\chi_{1,k}(\xi)=h\left(\frac{\varrho_k^\mu}3\left(\frac\xi n-1\right)\right).
\eeqs
By definition \eqref{h}, we have that
\beqs
\label{suppchi1k}
\supp \chi_{1,k}\subseteq \Big\{\xi\ :\ 
\Big\vert\frac \xi n-1\Big\vert\leq \frac3{2\varrho_k^\mu}\Big\}
\subseteq \left\{\xi\ :\ \vert\xi\vert\leq 3n\right\},
\eeqs
and
\beqs
\label{suppchi1}
\supp \left(1-\chi_{1,k}\right)\subseteq \Big\{\xi\ :\ 
\Big\vert\frac\xi n -1\Big\vert\geq \frac{3}{4\varrho_k^\mu}\Big\}.
\eeqs
This implies, by Lemma \ref{lemma2}, that
\beqs
\label{intvuota}
\supp (1-\chi_{1,k})\cap\supp w_{n,k}^{\alpha,\beta}=\emptyset.
\eeqs
Localizing now at frequency $n$
\beqs
\nonumber
D_x^rv_k^{\alpha,\beta}&=&\chi_{1,k}(D_x)D_x^rv_k^{\alpha,\beta}
+(1-\chi_{1,k}(D_x))D_x^rv_k^{\alpha,\beta}\\
\nonumber
&=&\chi_{1,k}(D_x)D_x^rv_k^{\alpha,\beta}+\sum_{j=0}^r\binom rj
(1-\chi_{1,k}(D_x))(D_x^jW_{n,k}^{\alpha,\beta})
D_x^{r-j}u_k,
\eeqs
and applying Calder\'on-Vaillancourt's Theorem \ref{thC}, we come to:
\beqs
\nonumber
\|D_x^rv_k^{\alpha,\beta}(t,\cdot)\|
\leq&&|\chi_{1,k}(\xi)\xi^r|_{2,2}^0\cdot\|v_k^{\alpha,\beta}\|\\
\label{C3}
&&+\sum_{j=0}^r\binom rj\varrho_k^j
\left |\sigma\left((1-\chi_{1,k}(D_x))W_{n,k}^{\alpha+j,\beta}
D_x^{r-j}\right)\right |_{2,2}^0\cdot\|u_k\|.
\eeqs
Note that $|\chi_{1,k}(\xi)\xi^r|_{2,2}^0\leq c_1n^r$ for some $c_1>0$,
because of \eqref{suppchi1k}; 
to estimate the second term of \eqref{C3}, by Theorem \ref{thA}
and \eqref{intvuota} we write, for every integer $\nu\geq1$:
\beqs
\nonumber
\hspace*{10mm}
&&\sigma\left((1-\chi_{1,k}(D_x))W_{n,k}^{\alpha+j,\beta}D_x^{r-j}\right)=
\sum_{0\leq\gamma\leq \nu-1}\frac{1}{\gamma!}\partial_\xi^\gamma(1-\chi_{1,k}(\xi))
D_x^\gamma(w_{n,k}^{\alpha+j,\beta}\xi^{r-j})\\
\nonumber
&&+\int_0^1\frac{(1-\theta)^{\nu-1}}{(\nu-1)!}\int\!\!\int e^{-iy\eta}
\partial_\xi^\nu(1-\chi_{1,k}(\xi+\theta\eta)D_x^\nu(w_{n,k}^{\alpha+j,\beta}
(t,x+y;\xi)
\xi^{r-j})dy\dbar\eta d\theta\\
\label{C2}
&&=\int_0^1\frac{(1-\theta)^{\nu-1}}{(\nu-1)!}
\Osc_\nu(t,\theta,x,\xi)
 d\theta,
\eeqs
where 
\beqsn
\Osc_\nu(t,\theta,x,\xi):=\int\!\!\int e^{-iy\eta}
\partial_\xi^\nu(1-\chi_{1,k}(\xi+\theta\eta)D_x^\nu w_{n,k}^{\alpha+j,\beta}
(t,x+y;\xi)
\xi^{r-j}dy\dbar\eta.
\eeqsn

Writing
$
\xi^{r-j}=\sum_{h=0}^{r-j}\binom{r-j}{h}(\xi+\theta\eta)^h(-\theta\eta)^{r-j-h}
$
and $e^{-iy\eta}(-\eta)^{r-j-h}=D_y^{r-j-h}e^{-iy\eta}$, we have, integrating 
by parts:
\beqsn
\Osc_\nu=&&-\sum_{h=0}^{r-j}\binom{r-j}{h}\theta^{r-j-h}\int\!\!\int
\partial_\xi^\nu\chi_{1,k}(\xi+\theta\eta)\cdot(\xi+\theta\eta)^h
D_x^\nu w_{n,k}^{\alpha+j,\beta}(t,x+y;\xi)\\
&&\cdot D_y^{r-j-h}e^{-iy\eta}dy\dbar\eta\\
=&&\sum_{h=0}^{r-j}(-1)^{r-j-h+1}\binom{r-j}{h}\theta^{r-j-h}\int\!\!\int
e^{-iy\eta}\partial_\xi^\nu\chi_{1,k}(\xi+\theta\eta)\cdot(\xi+\theta\eta)^h\\
&&\cdot
D_y^{\nu+r-j-h}w_{n,k}^{\alpha+j,\beta}(t,x+y;\xi)dy\dbar\eta\\
=&&\sum_{h=0}^{r-j}(-1)^{r-j-h+1}\binom{r-j}{h}\theta^{r-j-h}\varrho_h^{\nu+r-j-h}
\int\!\!\int
e^{-iy\eta}\partial_\xi^\nu\chi_{1,k}(\xi+\theta\eta)\cdot(\xi+\theta\eta)^h\\
&&\cdot
w_{n,k}^{\alpha+\nu+r-h,\beta}(t,x+y;\xi)dy\dbar\eta.
\eeqsn

By Theorem \ref{thB}, \eqref{chi1k}, \eqref{suppchi1k} and
Lemma \ref{lemmaAB1}, for $\theta
\in[0,1]$ we have that
\beqsn
|\Osc_\nu(t,\theta)|_{2,2}^0\leq&&\sum_{h=0}^{r-j}c_h
\varrho_k^{\nu+r-j-h}|\partial_\xi^\nu
\chi_{1,k}(\xi)\xi^h|_{4,4}^0\cdot|w_{n,k}^{\alpha+\nu+r-h,\beta}(t,x;\xi)|_{4,4}^0\\
\leq&&\sum_{h=0}^{r-j}C_{\alpha,\nu,r,h,\beta}\varrho_k^{\nu+r-j-h}
\left(\frac{\varrho_k^\mu}{n}\right)^\nu n^h
\varrho_k^{\frac12+4}\\
\leq
&&\sum_{h=0}^{r-j}C'_{\alpha,\beta,\nu,r,h}\left(\frac{\varrho_k^\mu}{n}\right)^\nu
\left(\frac{n}{\varrho_k}\right)^{r-j}\varrho_k^{\nu+r-j+\frac12+4}\\
=&&C_{\alpha,\beta,\nu,r,j}\,
\varrho_k^{\frac12+4}\left(\frac{\varrho_k^{\mu+1}}{n}\right)^\nu
n^{r-j}
\eeqsn
for some $c_h,C_{\alpha,\nu,r,h,\beta},C'_{\alpha,\beta,\nu,r,h},
C_{\alpha,\beta,\nu,r,j}>0$, since $(n/\varrho_k)^h\leq(n/\varrho_k)^{r-j}$
for $0\leq h\leq r-j$.

Substituting in \eqref{C2} and integrating with respect to $\theta$ we 
thus have that
\beqs\label{AB6}
\left |\sigma\left((1-\chi_{1,k}(D_x))W_{n,k}^{\alpha+j,\beta}
D_x^{r-j}\right)\right |_{2,2}^0\leq|\Osc_\nu|_{2,2}^0
\leq C_{\alpha,\beta,\nu,r,j}\varrho_k^{4+\frac12}
\left(\frac{\varrho_k^{\mu+1}}{n}\right)^\nu n^{r-j}.
\eeqs

Substituting in \eqref{C3}, and taking into account \eqref{C'T}, we have that
\beqsn
\|D_x^rv_k^{\alpha,\beta}(t,\cdot)\|\leq&&
c_1n^r\|v_k^{\alpha,\beta}\|+\sum_{j=0}^r
C'_{\alpha,\beta,\nu,r,j}\varrho_k^{4+\frac12+j}
\left(\frac{\varrho_k^{\mu+1}}{n}\right)^\nu
n^{r-j+q}\\
\leq&&c_1n^r\|v_k^{\alpha,\beta}\|+
C_{\alpha,\beta,\nu,r}\varrho_k^{4+\frac12}
\left(\frac{\varrho_k^{\mu+1}}{n}\right)^\nu
n^{r+q}
\eeqsn
for some $C'_{\alpha,\beta,\nu,r,j},C_{\alpha,\beta,\nu,r}>0$,
since $\left(\frac{\varrho_k}{n}\right)^j\leq1$ for every $j$.
\end{proof}

The following two lemmas give estimates of some pseudo-differential operators 
acting on the  functions $u_k$. 

\begin{Lemma}
\label{lemmaAB2}
Let $n=\varrho_k^a$ with $a\geq\mu$.
Then for every $\sigma,\gamma,\lambda\in\N_0$ 
the operators 
$W_{n,k}^{\sigma,\gamma}(t,x,D_x)$ satisfy
\beqs\label{ind}
W_{n,k}^{\sigma,\gamma}D_x^\lambda=\ds\sum_{j=0}^\lambda c_j
\varrho_k^j D_x^{\lambda-j}W_{n,k}^{\sigma+j,\gamma},
\eeqs
for some $c_0,\ldots,c_\lambda>0$.
Moreover, there are constants $C_\lambda>0$ and, for all $\nu\in\N_0$,
$C_{\sigma,\gamma,\lambda,\nu}>0$ such that for all $t\in[0,\varrho_k/n^{p-1}]$
with $k$ large enough:
\beqs
\label{2.28}
\ \ \|W_{n,k}^{\sigma,\gamma}(t,\cdot,D_x)D_x^\lambda u_k(t,\cdot)\|\leq
C_\lambda \sum_{j=0}^\lambda \varrho_k^jn^{\lambda-j}
\|v_k^{\sigma+j,\gamma}\|+
C_{\sigma,\gamma,\lambda,\nu}\varrho_k^{4+\frac12}\left(\frac{\varrho_k^{\mu+1}}{n}
\right)^\nu n^{\lambda+q}.
\eeqs
\end{Lemma}

\begin{proof}
Let us first prove \eqref{ind} by induction on $\lambda\in\N$.

For $\lambda=1$ we clearly have
$
W_{n,k}^{\sigma,\gamma}D_x=D_x W_{n,k}^{\sigma,\gamma}-\varrho_kW_{n,k}^{\sigma+1,\gamma}.
$

Let us assume \eqref{ind} to be true for every $\lambda'<\lambda$ and let us
prove it for $\lambda$. By Theorem \ref{thA}:
\beqsn
W_{n,k}^{\sigma,\gamma}D_x^\lambda=&&D_x^\lambda W_{n,k}^{\sigma,\gamma}
+[W_{n,k}^{\sigma,\gamma},
D_x^\lambda]\\
=&&D_x^\lambda W_{n,k}^{\sigma,\gamma}-\op\left(\sum_{\alpha=1}^\lambda
\frac{1}{\alpha!}\partial_\xi^\alpha\xi^\lambda\cdot
D_x^\alpha w_{n,k}^{\sigma,\gamma}\right)\\
=&&D_x^\lambda W_{n,k}^{\sigma,\gamma}-
\sum_{\alpha=1}^\lambda
\binom\lambda\alpha\varrho_k^\alpha\left(W_{n,k}^{\sigma+\alpha,\gamma}
D_x^{\lambda-\alpha}\right).
\eeqsn
By the inductive assumption, we thus have that
\beqsn
W_{n,k}^{\sigma,\gamma}D_x^\lambda
=&&D_x^\lambda W_{n,k}^{\sigma,\gamma}-\sum_{\alpha=1}^\lambda
\binom\lambda\alpha\varrho_k^\alpha\left(
\sum_{\ell=0}^{\lambda-\alpha}C_\ell\varrho_k^\ell D_x^{\lambda-\alpha-\ell}
W_{n,k}^{\sigma+\alpha+\ell,\gamma}\right)\\
=&&D_x^\lambda W_{n,k}^{\sigma,\gamma}-\sum_{\alpha=1}^\lambda\sum_{\ell=0}^{\lambda-\alpha}
C_{\alpha,\lambda,\ell}\varrho_k^{\alpha+\ell}D_x^{\lambda-\alpha-\ell}
W_{n,k}^{\sigma+\alpha+\ell,\gamma}\\
=&&\sum_{\alpha'=0}^\lambda C_{\alpha',\lambda}\varrho_k^{\alpha'}D_x^{\lambda-\alpha'}
W_{n,k}^{\sigma+\alpha',\gamma}.
\eeqsn
Therefore \eqref{ind} is proved and, applying
Lemma \ref{lemmaAB3} for $j\leq\lambda-1$, we have 
that for every $\nu\in\N$:
\beqsn
\|W_{n,k}^{\sigma,\gamma}(t,\cdot,D_x)D_x^\lambda u_k(t,\cdot)\|
\leq&&\ds\sum_{j=0}^\lambda c_j\varrho_k^j \|D_x^{\lambda-j}v_k^{\sigma+j,\gamma}\|\\
\leq&&\sum_{j=0}^{\lambda-1}c_j\varrho_k^j\left(c_1n^{\lambda-j}
\|v_k^{\sigma+j,\gamma}\|+
C_{\sigma,j,\gamma,\lambda,\nu}\varrho_k^{4+\frac12}\left(\frac{\varrho_k^{\mu+1}}{n}
\right)^\nu n^{\lambda-j+q}\right)\\
&&+c_\lambda\varrho_k^\lambda\|v_k^{\sigma+\lambda,\gamma}\|\\
\leq&&C'_\lambda \sum_{j=0}^\lambda \varrho_k^j n^{\lambda-j}
\|v_k^{\sigma+j,\gamma}\|
+C_{\sigma,\gamma,\lambda,\nu}\varrho_k^{4+\frac12}\left(\frac{\varrho_k^{\mu+1}}{n}
\right)^\nu n^{\lambda+q},
\eeqsn
for some $C'_\lambda,C_{\sigma,\gamma,\lambda,\nu}>0$.
This proves \eqref{2.28}.
\end{proof}

\begin{Lemma}
\label{lemmaAB4}
Let $a_j=a_j(t,x)$, for $0\leq j\leq p-1$, be the coefficients of the
operator \eqref{op}, and let $n=\varrho_k^a$ with
$a\geq\mu+1\geq2$.
Then, for every $\nu\in\N$ there exists $C_\nu>0$ such that
\beqsn
\|[a_j,W_{n,k}]D_x^ju_k(t,\cdot)\|\leq&&
C_\nu\,n^j\sum_{1\leq\alpha_1+\alpha_2\leq (\nu-1)(p-1)+j}
\left(\frac{\varrho_k^\mu}{n}
\right)^{\alpha_1+\alpha_2}
\|v_k^{\alpha_1,\alpha_2}\|
+C_\nu\varrho_k^{4+\frac12}\left(\frac{\varrho_k^{\mu+1}}{n}\right)^\nu n^{q+j}
\eeqsn
for all $t\in[0,\varrho_k/n^{p-1}]$ with $k$ large enough.
\end{Lemma}

\begin{proof}
By Theorem \ref{thA}, for all $\nu\in\N$
\beqs
\nonumber
&&\sigma\left([a_j(t,x),W_{n,k}(t,x,D_x)]D_x^j\right)=
\sigma([a_j,W_{n,k}])\cdot\xi^j\\
\label{A5}
=&&-\left(\sum_{1\leq\alpha\leq \nu-1}\frac{1}{\alpha!}\partial_\xi^\alpha
w_{n,k}\cdot D_x^\alpha a_j\right)\xi^j
-\int_0^1\frac{(1-\theta)^{\nu-1}}{(\nu-1)!}
\tilde{\Osc}_\nu(t,\theta,x,\xi)
 d\theta,
\eeqs
where
\beqsn
\tilde{\Osc}_\nu(t,\theta,x,\xi):=
\int\!\!\int
e^{-iy\eta}\partial_\xi^\nu w_{n,k}(t,x;\xi+\theta\eta)D_x^\nu
a_j(t,x+y)\cdot\xi^jdy\dbar\eta.
\eeqsn 

Arguing as in the proof of Lemma \ref{lemmaAB3} we can estimate,
by Theorem \ref{thB} and Lemma \ref{lemmaAB1}:
\beqs
\nonumber
|\tilde{\Osc}_\nu|_{2,2}^{0}=&&\left|
\sum_{h=0}^j\binom jh \theta^{j-h}\int\!\!\int
\partial_\xi^\nu w_{n,k}(t,x;\xi+\theta\eta)\cdot(\xi+\theta\eta)^h
D_x^\nu a_j(t,x+y)D_y^{j-h}e^{-iy\eta}dy\dbar\eta\right|_{2,2}^{0}\\
\nonumber
\leq&&\sum_{h=0}^j\binom jh\left|\int\!\!\int e^{-iy\eta}
\partial_\xi^\nu w_{n,k}(t,x;\xi+\theta\eta)\cdot(\xi+\theta\eta)^h
D_x^{\nu+j-h}a_j(t,x+y)dy\dbar\eta\right|_{2,2}^{0}\\
\nonumber
\leq&&\sum_{h=0}^jC_j\left|\xi^h\partial_\xi^\nu w_{n,k}(t,x;\xi)\right|_{4,4}^{0}
\cdot\left|D_x^{\nu+j-h}a_j(t,x)\right|_{4,4}^{0}\\
\label{A4}
\leq&&C_{\nu}n^j
\varrho_k^{4+\frac12}\left(\frac{\varrho_k^\mu}{n}\right)^\nu
\eeqs
for some $C_j,C_\nu>0$, since $a_j\in C([0,T];\B^\infty)$ for
$0\leq j\leq p-1$.

In order to estimate now the first term of \eqref{A5},
we previously compute, by the 
Fa\`a di Bruno formula:
\beqsn
\partial_\xi^\alpha
w_{n,k}
=&&\sum_{\alpha_1+\alpha_2=\alpha}\frac{\alpha!}{\alpha_1!\alpha_2!}
\varrho_k^{1/2}
\cdot\partial_\xi^{\alpha_1}h\Big(\varrho_k(x-x_k-pA_p(t)\xi^{p-1})\Big)
\cdot\partial_\xi^{\alpha_2}h
\Big(\varrho_k^\mu\Big(\frac\xi n-1\Big)\Big)\\
&&=\varrho_k^{1/2}
h\Big(\varrho_k(x-x_k-pA_p(t)\xi^{p-1})\Big)
\cdot\partial_\xi^{\alpha}h
\Big(\varrho_k^\mu\Big(\frac\xi n-1\Big)\Big)\\
&&+\sum_{\afrac{\alpha_1+\alpha_2=\alpha}{\alpha_1\geq1}}
\frac{\alpha!}{\alpha_1!\alpha_2!}
\sum_{\afrac{r_1+\ldots+r_s=\alpha_1}{r_h\geq1}}C_{s,r}\varrho_k^{1/2}
h^{(s)}\Big(\varrho_k(x-x_k-pA_p(t)\xi^{p-1})\Big)\\
&&\cdot\partial_\xi^{r_1}\Big[\varrho_k(x-x_k-pA_p(t)\xi^{p-1})\Big]
\cdots\partial_\xi^{r_s}\Big[\varrho_k(x-x_k-pA_p(t)\xi^{p-1})\Big]\\
&&\cdot\Big(\frac{\varrho_k^\mu}{n}\Big)^{\alpha_2}
h^{(\alpha_2)}\Big(\varrho_k\Big(\frac\xi n-1\Big)\Big)\\
=&&\left(\frac{\varrho_k^\mu}{n}\right)^\alpha\varrho_k^{1/2}
h\Big(\varrho_k(x-x_k-pA_p(t)\xi^{p-1})\Big)
h^{(\alpha)}
\Big(\varrho_k^\mu\Big(\frac\xi n-1\Big)\Big)\\
&&+\sum_{\afrac{\alpha_1+\alpha_2=\alpha}{\alpha_1\geq1}}
\frac{\alpha!}{\alpha_1!\alpha_2!}
\sum_{\afrac{r_1+\ldots+r_s=\alpha_1}{1\leq r_h\leq p-1}}C'_{s,r}
\Big(\varrho_kA_p(t)\Big)^{\alpha_1}\cdot\xi^{s(p-1)-\alpha_1}\\
&&\cdot
\Big(\frac{\varrho_k^\mu}{n}\Big)^{\alpha_2}
\varrho_k^{1/2}
h^{(s)}\Big(\varrho_k(x-x_k-pA_p(t)\xi^{p-1})\Big)
h^{(\alpha_2)}\Big(\varrho_k\Big(\frac\xi n-1\Big)\Big)
\eeqsn
for some $C_{s,r},C'_{s,r}>0$.
Coming back to the first term of \eqref{A5} and taking into account 
the definition \eqref{defwnk} of $w_{n,k}$:
\beqsn
&&\left(\sum_{1\leq\alpha\leq \nu-1}\frac{1}{\alpha!}\partial_\xi^\alpha
w_{n,k}\cdot D_x^\alpha a_j\right)\xi^j
\leq\sum_{1\leq\alpha\leq \nu-1}
\frac{D_x^\alpha a_j}{\alpha!}\left(\frac{\varrho_k^\mu}{n}\right)^{\alpha}
w_{n,k}^{0,\alpha}\cdot\xi^j\\
&&+\sum_{1\leq\alpha\leq \nu-1}\sum_{\afrac{\alpha_1+\alpha_2=\alpha}{\alpha_1\geq1}}
\frac{D_x^\alpha a_j}{\alpha_1!\alpha_2!}
\sum_{\afrac{r_1+\ldots+r_s=\alpha_1}{1\leq r_h\leq p-1}}C'_{s,r}
\varrho_k^{\alpha_1}A_p(t)^{\alpha_1}
\left(\frac{\varrho_k^\mu}{n}\right)^{\alpha_2}
w_{n,k}^{s,\alpha_2}\cdot\xi^{s(p-1)-\alpha_1+j}
\eeqsn
and hence
\beqsn
&&\Big\|\op\Big[\Big(\sum_{1\leq\alpha\leq \nu-1}\frac{1}{\alpha!}
\partial_\xi^\alpha
w_{n,k}\cdot D_x^\alpha a_j\Big)\xi^j\Big]u_k\Big\|\\
\nonumber
\leq&&\sum_{1\leq\alpha\leq \nu-1}C_{\alpha,j}
\left(\frac{\varrho_k^\mu}{n}\right)^{\alpha}
\left\|W_{n,k}^{0,\alpha}D_x^ju_k\right\|\\
\nonumber
&&+\sum_{1\leq\alpha\leq \nu-1}\sum_{\afrac{\alpha_1+\alpha_2=\alpha}{\alpha_1\geq1}}
C_{\alpha,j}
\sup_{[0,\varrho_k/n^{p-1}]}|A_p(t)|^{\alpha_1}
\varrho_k^{\alpha_1}\left(\frac{\varrho_k^\mu}{n}\right)^{\alpha_2}
\sum_{s=1}^{\alpha_1}\left\|W_{n,k}^{s,\alpha_2}D_x^{s(p-1)-\alpha_1+j}u_k\right\|.
\eeqsn

Applying \eqref{2.28}, since $s\leq\alpha_1$, $\mu\geq2$ and $0\leq j\leq p-1$, 
we thus obtain, for $t\in[0,\varrho_k/n^{p-1}]$:
\beqs
\nonumber
\hspace*{12mm}
&&\Big\|\op\Big[\Big(\sum_{1\leq\alpha\leq \nu-1}\frac{1}{\alpha!}
\partial_\xi^\alpha
w_{n,k}\cdot D_x^\alpha a_j\Big)\xi^j\Big]u_k\Big\|\\
\nonumber
\leq&&
\sum_{1\leq\alpha\leq \nu-1}C'_{\alpha,j}
\left(\frac{\varrho_k^\mu}{n}\right)^{\alpha}
\bigg[\sum_{h=0}^j\varrho_k^h n^{j-h}
\|v_k^{h,\alpha}\|+C_{\alpha,j,\nu}\varrho_k^{4+\frac12}
\left(\frac{\varrho_k^{\mu+1}}{n}\right)^\nu n^{j+q}\bigg]\\
\nonumber
&&+\sum_{1\leq\alpha\leq \nu-1}\sum_{\afrac{\alpha_1+\alpha_2=\alpha}{\alpha_1\geq1}}
C'_{\alpha,j}\left(\frac{\varrho_k}{n^{p-1}}\right)^{\alpha_1}
\varrho_k^{\alpha_1}\left(\frac{\varrho_k^\mu}{n}\right)^{\alpha_2}\\
\nonumber
&&\cdot
\sum_{s=1}^{\alpha_1}\bigg[C_{s,\alpha,j}
\sum_{h=0}^{s(p-1)-\alpha_1+j}\varrho_k^h n^{s(p-1)-\alpha_1+j-h}
\|v_k^{s+h,\alpha_2}\|\\
\nonumber
&&
+C_{s,\alpha,j,\nu}\varrho_k^{4+\frac12}
\left(\frac{\varrho_k^{\mu+1}}{n}\right)^\nu
n^{s(p-1)-\alpha_1+j+q}\bigg]\\
\nonumber
\leq&&C_\nu n^j\sum_{1\leq\alpha\leq \nu-1}
\left(\frac{\varrho_k^\mu}{n}\right)^{\alpha}
\sum_{h=0}^j\left(\frac{\varrho_k}{n}\right)^h
\|v_k^{h,\alpha}\|
+C_\nu\varrho_k^{4+\frac12}
\left(\frac{\varrho_k^{\mu+1}}{n}\right)^\nu n^{j+q}\\
\nonumber
&&+C_\nu\sum_{1\leq\alpha\leq \nu-1}
\sum_{\afrac{\alpha_1+\alpha_2=\alpha}{\alpha_1\geq1}}
\left(\frac{\varrho_k^\mu}{n}\right)^{\alpha_1+\alpha_2}\frac{1}{n^{\alpha_1(p-2)}}
\\
\nonumber
&&\cdot\bigg[n^{\alpha_1(p-2)+j}\sum_{s=1}^{\alpha_1}
\sum_{h=0}^{s(p-1)-\alpha_1+j}\left(\frac{\varrho_k}{n}\right)^h
\|v_k^{s+h,\alpha_2}\|
+\varrho_k^{4+\frac12}
\left(\frac{\varrho_k^{\mu+1}}{n}\right)^\nu n^{\alpha_1(p-2)+j+q}\bigg]\\
\nonumber
\leq&&C_\nu n^j\sum_{1\leq\alpha_2\leq \nu-1}\sum_{h=0}^j
\left(\frac{\varrho_k^\mu}{n}\right)^{h+\alpha_2}\|v_k^{h,\alpha_2}\|
+C'_{\nu}\varrho_k^{4+\frac12}
\left(\frac{\varrho_k^{\mu+1}}{n}\right)^\nu n^{q+j}\\
\nonumber
&&+C_\nu n^j\sum_{\afrac{1\leq \alpha_1+\alpha_2\leq \nu-1}{\alpha_1\geq1}}
\sum_{s=1}^{\alpha_1}\sum_{h=0}^{s(p-1)-\alpha_1+j}
\left(\frac{\varrho_k^\mu}{n}\right)^{s+h+\alpha_2}
\|v_k^{s+h,\alpha_2}\|\\
\label{A6}
\leq&&C_\nu n^j\sum_{1\leq \alpha_1+\alpha_2\leq (\nu-1)(p-1)+j}
\left(\frac{\varrho_k^\mu}{n}\right)^{\alpha_1+\alpha_2}
\|v_k^{\alpha_1,\alpha_2}\|
+C'_{\nu}\varrho_k^{4+\frac12}
\left(\frac{\varrho_k^{\mu+1}}{n}\right)^\nu n^{q+j}
\eeqs
for some $C_{\alpha,j},C_{\alpha,j,\nu},C'_{\alpha,j},C'_{\alpha,j,\nu},
C_{s,\alpha,j},C_{s,\alpha,j,\nu},C_\nu,C'_{\nu}>0$.

By the Calder\'on-Vaillancourt's Theorem \ref{thC}, by 
\eqref{A5}, \eqref{A4} and \eqref{A6} we get:
\beqsn
&&\|[a_j,W_{n,k}]D_x^ju_k\|\leq C|\tilde{\Osc}_\nu|_{2,2}^{0}\cdot\|u_k\|
+\Big\|\op\Big[\Big(\sum_{1\leq\alpha\leq \nu-1}
\frac{1}{\alpha!}\partial_\xi^\alpha w_{n,k}\cdot D_x^\alpha a_j\Big)
\xi^j\Big]u_k\Big\|\\
\leq&&C'_{\nu}\varrho_k^{4+\frac12}
\left(\frac{\varrho_k^{\mu+1}}{n}\right)^\nu n^{q+j}
+C_{\nu}n^j\sum_{1\leq \alpha_1+\alpha_2\leq (\nu-1)(p-1)+j}
\left(\frac{\varrho_k^\mu}{n}\right)^{\alpha_1+\alpha_2}
\|v_k^{\alpha_1,\alpha_2}\|\\
\eeqsn
for some $C,C_\nu,C'_\nu>0$.
\end{proof}


\section{Estimates from below}
\label{sec3}

In this section we want to produce estimates from below of the $L^2$-norms 
of the functions $v_k$ and $v_k^{\alpha,\beta}$, and then of a 
linear combination $\sigma_k(t)$ of the $L^2$-norms of $v_k^{\alpha,\beta}$, 
$\alpha+\beta\geq0.$  

We start with the estimate of $\|v_k(0,\cdot)\|$.
For $n$ as in \eqref{n} and $k$ large enough, from \eqref{fgk} we have that
\beqs
\label{hint}
\supp \hat g_k=\supp \hat\psi(\xi-n)
\subseteq\{\xi\in\R:\ 
h(\varrho_k^\mu(\xi/ n-1))=1\}.
\eeqs
Therefore
\beqsn
v_k(0,x)&=&W_{n,k}u_k(0,x)=\int e^{ix\xi}w_{n,k}(0,x,\xi)\widehat {g_k}(\xi)
\dbar\xi
\\
\nonumber
&=&\int e^{ix\xi} \varrho_k^{1/2}h\left(\varrho_k(x-x_k)\right)
\underbrace{h\left(\varrho_k^\mu(\xi/n-1)\right)}_1e^{-ix_k\xi}
\hat\psi(\xi-n)\dbar\xi
\\
\nonumber
&=& \varrho_k^{1/2}h\left(\varrho_k(x-x_k)\right)e^{i(x-x_k)n}\psi(x-x_k)
\eeqsn
and 
\beqs\label{vk0}
\|v_k(0,\cdot)\|^2&=&\int \varrho_k|h\left(\varrho_k(x-x_k)\right)|^2
|\psi(x-x_k)|^2dx=\int |h(y)|^2|\psi(y/\varrho_k)|^2dy
\\
\nonumber
&\geq &\int |h(y)|^2dy=\|h\|^2>0
\eeqs
if $k$ is large enough, since $\psi(0)=2$ and $\varrho_k\to+\infty$.

Now, to produce an estimate from below of $\|v_k(t,\cdot)\|$, our 
idea is to follow the energy method, producing a "reverse energy estimate". 
To this aim, denoting by $\langle\cdot,\cdot\rangle$ 
the scalar product on $L^2$, 
we consider
\beqs
\nonumber
\ds\frac{d}{dt}\|v_k(t,\cdot)\|^2=&&
2\Re \langle \partial_tv_k,v_k\rangle\\
\label{encontrary}
=&& 2\Re i \langle Pv_k,v_k\rangle-2\Re i 
a_p(t)\langle D_x^pv_k,v_k\rangle
-2\Re i \sum_{j=0}^{p-1}\langle (a_j(t,x)D_x^jv_k,v_k\rangle.
\eeqs
We compute separately estimates from below of each term in formula 
\eqref{encontrary}.
By definition of $v_k$ we have that
\beqsn
Pv_k&=&PW_{n,k}u_k=W_{n,k}Pu_k+[P, W_{n,k}]u_k= 
\\
&=&0+[D_t+a_p(t)D_x^p, W_{n,k}]u_k+\sum_{j=0}^{p-1}[a_j(t,x)D_x^j, W_{n,k}]u_k,
\eeqsn
since $Pu_k=0$.

Developing the symbol of the commutator $[D_t+a_p(t)D_x^p, W_{n,k}]$ and using 
the fact that $w_{n,k}$ is the solution of Hamilton's
equation \eqref{215} we obtain, by Theorem \ref{thA}:
\beqsn
\sigma\left(\left[D_t+a_p(t)D_x^p, W_{n,k}\right]\right)(t,x,\xi)=&&
D_tw_{n,k}+a_p(t)\sigma\left([D_x^p, W_{n,k}]\right)\\
=&&D_tw_{n,k}+a_p(t)\sum_{\alpha=1}^p\frac{1}{\alpha!}\partial_\xi^\alpha
\xi^p\cdot D_x^\alpha w_{n,k}\\
=&&(D_t+pa_p(t)\xi^{p-1}D_x)w_{n,k}+a_p(t)\sum_{\alpha=2}^p\binom p\alpha
\xi^{p-\alpha}D_x^\alpha w_{n,k}\\
=&&a_p(t)\sum_{\alpha=2}^p\binom p\alpha
\xi^{p-\alpha}D_x^\alpha w_{n,k}.
\eeqsn
Defining then
\beqs
\label{fk}
f_k:=\op\left(a_p(t)\sum_{\alpha=2}^p\binom p\alpha
\xi^{p-\alpha}D_x^\alpha w_{n,k}\right)u_k
+\sum_{j=0}^{p-1}[a_j(t,x)D_x^j,W_{n,k}]u_k,
\eeqs
we have that
\beqsn
Pv_k=f_k
\eeqsn
and hence from \eqref{encontrary} we get
\beqs
\nonumber
\ds\frac{d}{dt}\|v_k(t,\cdot)\|^2=&&2\Re i\langle f_k,v_k\rangle
-2\Re ia_p(t)\langle D_x^pv_k,v_k\rangle-\sum_{j=0}^{p-1}2\Re i
\langle a_jD_x^jv_k,v_k\rangle\\
\label{EB3}
=&&2\Re i\langle f_k,v_k\rangle
-\sum_{j=0}^{p-1}\langle (ia_jD_x^j+(ia_jD_x^j)^*)v_k,v_k\rangle
\eeqs
since $\Re i\langle D_x^p v_k,v_k\rangle=0$.
Now,
\beqsn
\sigma(ia_j(t,x)D_x^j)^*=\sum_{\alpha\geq0}\frac{1}{\alpha!}\partial_\xi^\alpha
D_x^\alpha(\overline{ia_j(t,x)\xi^j})
=\sum_{\alpha=0}^j\binom j\alpha D_x^\alpha(
-i\Re a_j-\Im a_j(t,x))\xi^{j-\alpha},
\eeqsn
and hence
\beqsn
\sum_{j=0}^{p-1}\sigma[(ia_jD_x^j)+(ia_jD_x^j)^*]=&&
\sum_{j=0}^{p-1}\left[-2\Im a_j\xi^j
+\sum_{\alpha=1}^j\binom j\alpha D_x^\alpha\left(-i\Re a_j-
 \Im a_j\right)\xi^{j-\alpha}\right]\\
=&&-2\sum_{j=0}^{p-1}\Im a_j\xi^j+
\sum_{h=0}^{p-2}\sum_{j=h+1}^{p-1}\binom jh D_x^{j-h}\left(-i\Re a_j-\Im a_j
\right)\xi^h\\
=&&-2\Im a_{p-1}\xi^{p-1}\\
&&+\sum_{h=0}^{p-2}\left[-2\Im a_h
+\sum_{j=h+1}^{p-1}\binom jh D_x^{j-h}\left(-i\Re a_j-\Im a_j\right)\right]\xi^h.
\eeqsn

Substituting in \eqref{EB3}, 
we have that there exist
postive constants $A_1,c'$ such that
\beqs
\nonumber
\ds\frac{d}{dt}\|v_k(t,\cdot)\|^2\geq &&-2\|f_k\|\cdot\|v_k\|
+2\langle\Im a_{p-1} D_x^{p-1}v_k,v_k\rangle-A_1\|v_k\|^2\\
\nonumber
&&+\sum_{h=1}^{p-2}
\left[2\langle\Im a_hD_x^h v_k,v_k\rangle
+\sum_{j=h+1}^{p-1}
\binom jh
\langle(D_x^{j-h}(i\Re a_j+\Im a_j))D_x^hv_k,v_k\rangle\right]\\
\label{encontrary2}
\geq&&2\langle\Im a_{p-1} D_x^{p-1}v_k,v_k\rangle
-2\|f_k\|\cdot\|v_k\|-A_1\|v_k\|^2-
c'\frac{\,n^{p-1}}{\varrho_k}
\|v_k\|^2,
\eeqs
since 
\beqsn
|\langle\Im a_hD_x^hv_k,v_k\rangle|
\leq c n^h\|v_k\|^2
\leq c n^{p-2}\|v_k\|^2
\leq c\frac{\,n^{p-1}}{\varrho_k}\|v_k\|^2
\eeqsn
because of the support of $w_{n,k}$,
and analogously
\beqsn
|\langle(D_x^{j-h}(i\Re a_j+\Im a_j))D_x^hv_k,v_k\rangle|
\leq c\frac{\,n^{p-1}}{\varrho_k}\|v_k\|^2.
\eeqsn

 Now we want to give
 estimates of the terms in \eqref{encontrary2}. This is done in the 
following Propositions \ref{ima2} and \ref{lemma24}.
 
\begin{Prop}
\label{ima2} 
Let $n=\varrho_k^a$ with $a\geq\mu\geq 2$. 
Then, for all 
$\nu\in\N$ there exists $C_\nu>0$ such that, for every
$t\in \left[0,\ds\frac{\varrho_k}{n^{p-1}}\right]$ with
$k$ large enough:
\beqs
\label{AB8}
\langle \Im a_{p-1}(t,x)D_x^{p-1}v_k,v_k\rangle&\geq&
\left( \Im a_{p-1}(t,x_k+pA_p(t)n^{p-1})n^{p-1}-
C\frac{\,n^{p-1}}{\varrho_k}\right)\|v_k\|^2
\\
\nonumber
&&-C_\nu{\varrho_k}^{4+\frac12}\left(\frac{{\varrho_k}^{\mu+1}}{n}\right)^\nu 
n^{q+p-1}\|v_k\|,
\eeqs
for some fixed $C>0.$
\end{Prop}
\begin{proof}
We split 
\beqs
\nonumber
\Im a_{p-1}(t,x)D_x^{p-1}&=&\Im a_{p-1}(t,x_k +pA_p(t)n^{p-1})n^{p-1}\\
\label{split}
&&+\Im a_{p-1}(t,x_k +pA_p(t)n^{p-1})(D_x^{p-1}-n^{p-1})\\
\nonumber
&&+\left(\Im a_{p-1}(t,x)-\Im a_{p-1}(t,x_k +pA_p(t)n^{p-1})\right)D_x^{p-1} 
\eeqs
and set
\beqsn
&I_1:=\Im a_{p-1}(t,x_k +pA_p(t)n^{p-1})n^{p-1},&
\\
&I_2:=\Im a_{p-1}(t,x_k +pA_p(t)n^{p-1})(D_x^{p-1}-n^{p-1})&
\\
&I_3:= \left(\Im a_{p-1}(t,x)-\Im a_{p-1}(t,x_k +pA_p(t)n^{p-1})\right)D_x^{p-1}.& 
\eeqsn
We have 
\beqs\label{I1vk}
\langle I_1v_k,v_k\rangle=\Im a_{p-1}(t,x_k +pA_p(t)n^{p-1})n^{p-1}\|v_k\|^2.
\eeqs
To estimate $\langle I_2v_k,v_k\rangle$, we localize at frequency
$n$ by means of the
function $\chi_{1,k}$ defined in \eqref{chi1k}
and write 
\beqsn
I_2v_k&=&\chi_{1,k}(D_x)I_2v_k+(1-\chi_{1,k}(D_x))I_2v_k
\\
&=&\Im a_{p-1}(t,x_k +pA_p(t)n^{p-1})[\chi_{1,k}(D_x)(D_x^{p-1}-n^{p-1})v_k\\
&&+\left(1-\chi_{1,k}(D_x)\right)(D_x^{p-1}-n^{p-1})v_k],
\eeqsn
so, denoting by
\beqs
&J_1:=\|\chi_{1,k}(D_x)(D_x^{p-1}-n^{p-1})v_k\|,&
\\
\label{j2}
&J_2:=\|\left(1-\chi_{1,k}(D_x)\right)(D_x^{p-1}-n^{p-1})v_k\|,&
\eeqs
we have 
\beqs\label{normaI2vk}
\|I_2v_k\|\leq |\Im a_{p-1}(t,x_k +pA_p(t)n^{p-1})|(J_1+J_2).
\eeqs
By Calder\'on-Vaillancourt's Theorem \ref{thC},
\beqs
\label{normaj1}
J_1\leq C\vert\chi_{1,k}(\xi)(\xi^{p-1}-n^{p-1})\vert_{2,2}^{0}\|v_k\|
\leq C'\frac{n^{p-1}}{\varrho_k^\mu}\|v_k\|
\eeqs
for some $C,C'>0$, since by \eqref{suppchi1k}:
\beqsn
|\chi_{1,k}(\xi)(\xi^{p-1}-n^{p-1})|=&&
|\chi_{1,k}(\xi)(\xi-n)(\xi^{p-2}+n\xi^{p-3}+n^2\xi^{p-4}+\ldots+
n^{p-2})|\\
\leq&& c\frac{n}{\varrho_k^\mu}(p-1)n^{p-2}=c'\frac{n^{p-1}}{\varrho_k^\mu},
\eeqsn
for some $c,c'>0$, and for all $\gamma=\gamma_1+\gamma_2$ 
with $|\gamma|\leq 2$ there are constants $C_{\gamma_1}$, $C_\gamma>0$ such that:
\beqsn
|\partial_\xi^{\gamma_1}\chi_{1,k}(\xi)
\partial_\xi^{\gamma_2}(\xi^{p-1}-n^{p-1})|\leq 
\begin{cases}
 C_{\gamma_1} \ds\frac{n^{p-1}}{\varrho_k^\mu}& \gamma_2=0
\\
C_{\gamma}n^{p-1-\gamma_2}\leq  C_\gamma\ds\frac{n^{p-1}}{\varrho_k^\mu}&
\gamma_2\geq1.
\end{cases}
\eeqsn
As it concerns \eqref{j2}, by definition of $v_k$ we write
\beqs
\label{j22}
(D_x^{p-1}-n^{p-1})v_k=
\left(W_{n,k}(D_x^{p-1}-n^{p-1})+[D_x^{p-1}-n^{p-1},W_{n,k}]\right)u_k.
\eeqs
Since
$
\sigma([D_x^{p-1}-n^{p-1},W_{n,k}])=
\sum_{\alpha=1}^{p-1}\binom{p-1}{\alpha}\xi^{p-1-\alpha}
\varrho_k^\alpha w_{n,k}^{\alpha,0},
$
we have that
\beqs
\label{j222}
[D_x^{p-1}-n^{p-1},W_{n,k}]=
\sum_{\alpha=1}^{p-1}\binom{p-1}{\alpha}
\varrho_k^\alpha W_{n,k}^{\alpha,0}D_x^{p-1-\alpha}
\eeqs
and therefore, by \eqref{j2}, \eqref{j22}, \eqref{j222}, the
Calder\'on-Vaillancourt's Theorem \ref{thC} and \eqref{AB6}, for every
$\nu\in\N$ there are constants $C,C'_\nu,C''_\nu>0$ such that:
\beqs
\label{normaj2}
\nonumber
J_2&\leq &\|\left(1-\chi_{1,k}(D_x)\right)W_{n,k}(D_x^{p-1}-n^{p-1})u_k\|\\
\nonumber
&&+\sum_{\alpha=1}^{p-1}\binom{p-1}{\alpha}\varrho_k^\alpha
\|(1-\chi_{1,k}(D_x))W_{n,k}^{\alpha,0}D_x^{p-1-\alpha}u_k\|
\\
\nonumber
&\leq & C\bigg(\left |\sigma\left(\left(1-\chi_{1,k}(D_x)\right)
W_{n,k}(D_x^{p-1}-n^{p-1})\right)\right |_{2,2}^{0}\\
\nonumber
&&+\sum_{\alpha=1}^{p-1}\varrho_k^\alpha
\left |\sigma\left(\left(1-\chi_{1,k}(D_x)\right)W_{n,k}^{\alpha,0}D_x^{p-1-\alpha}
\right)
\right |_{2,2}^{0}\bigg)\|u_k\|
\\
\nonumber
&\leq& C'_\nu\varrho_k^{4+\frac12}\left(\frac{\varrho_k^{\mu+1}}{n}\right)^\nu
(n^{p-1}+\varrho_k n^{p-2}+\ldots+\varrho_k^{p-1})n^q
\\
&\leq& C''_\nu\varrho_k^{4+\frac12}\left(\frac{\varrho_k^{\mu+1}}{n}\right)^\nu 
n^{q+p-1}.
\eeqs
Substituting now \eqref{normaj1} and \eqref{normaj2} 
in \eqref{normaI2vk} we come to
\beqs\label{3.7}
\|I_2v_k\|\leq C\ds\frac{n^{p-1}}{\varrho_k^\mu}\|v_k\|+
C_\nu\varrho_k^{4+\frac12}\left(\frac{\varrho_k^{\mu+1}}{n}\right)^\nu n^{q+p-1}
\eeqs
for some $C,C_\nu>0$, and hence
\beqs\label{I2vk}
\langle I_2v_k,v_k\rangle\geq -C\ds\frac{n^{p-1}}{\varrho_k^\mu}\|v_k\|^2
-C_\nu\varrho_k^{4+\frac12}\left(\frac{\varrho_k^{\mu+1}}{n}\right)^\nu n^{q+p-1}
\|v_k\|.
\eeqs
Finally, to estimate $\langle I_3v_k,v_k\rangle$, we localize in a
neighborhood of $x_k+pA_p(t)\xi^{p-1}$ by defining,
for $h$ as 
in \eqref{h}, the function 
\beqs\label{chi2k}
\chi_{2,k}(x):= h\left(\varrho_k\frac{x-x_k-pA_p(t)\xi^{p-1}}{4pc_p}\right),
\eeqs 
where $c_p$ is the constant defined in Lemma \ref{lemma2}.
We have that
\beqs\label{suppchi2k}
\supp \chi_{2,k}\subseteq \left\{x\ :\ \vert x-x_k-pA_p(t)\xi^{p-1}\vert\leq 
\frac{2pc_p}{\varrho_k}\right\}
\eeqs
and
\beqs\label{supp1-chi2k}
\supp \left(1-\chi_{2,k}\right)\subseteq \left\{x\ :\ 
\vert\ x-x_k-pA_p(t)\xi^{p-1}\vert\geq \frac{pc_p}{\varrho_k}\right\}.
\eeqs
We now claim that 
\beqs\label{yuppi}
\supp (1-\chi_{2,k})\cap\supp W_{n,k}^{\alpha,\beta}=\emptyset
\qquad\forall t\in 
\left[0,\frac{\varrho_k}{n^{p-1}}\right].
\eeqs
This holds true because on the support of 
$w_{n,k}^{\alpha,\beta}$, given by Lemma \ref{lemma2}, we have that, for all 
$t\in \left[0,\ds\frac{\varrho_k}{n^{p-1}}\right]$,
\beqsn
\vert x-x_k-pA_p(t)\xi^{p-1}\vert&\leq&\vert x-x_k-pA_p(t)n^{p-1}\vert
+p|A_p(t)|\vert \xi^{p-1}-n^{p-1}\vert
\\
&\leq& \frac{c_p}{\varrho_k}+p\sup_{[0,T]}|a_p|\cdot t\cdot |\xi-n|
\cdot|\xi^{p-2}+n\xi^{p-3}+\ldots+n^{p-2}|
\\
&\leq&  \frac{c_p}{\varrho_k}+c_p \frac{\varrho_k}{n^{p-1}}
\frac{n}{2\varrho_k^\mu}(p-1)n^{p-2}
\leq p\frac{c_p}{\varrho_k},
\eeqsn
by the definition of $c_p$. 
Therefore \eqref{yuppi} is proved and
\beqsn
I_3v_k=(1-\chi_{2,k}(x))I_3v_k+\chi_{2,k}(x)I_3v_k
=\chi_{2,k}(x)I_3v_k.
\eeqsn

Then, by Lemma \ref{lemmaAB3}:
\beqsn
\|I_3v_k\|=&&\|\chi_{2,k}(x)I_3v_k\|=
|\Im a_{p-1}(t,x)-\Im a_{p-1}(t,x_k +pA_p(t)n^{p-1})|
\cdot\|\chi_{2,k}(x)D_x^{p-1}v_k\|
\\
\leq &&\Big(\sup_{[0,T]\times\R} |\Im \partial_xa_{p-1}(t,x)|\Big)
\cdot |x-x_k-pA_p(t)n^{p-1}|
\cdot\|\chi_{2,k}(x)D_x^{p-1}v_k\|\\
\leq&&\frac{c}{\varrho_k}\|D_x^{p-1}v_k\|
\leq\frac{c}{\varrho_k}\left(c_1n^{p-1}\|v_k\|+C_\nu\varrho_k^{4+\frac12}
\left(\frac{\varrho_k^{\mu+1}}{n}\right)^\nu n^{q+p-1}\right)
\\
\leq&& C'\frac{\,n^{p-1}}{\varrho_k}\|v_k\|+C'_\nu{\varrho_k}^{3+\frac12}
\left(\frac{{\varrho_k}^{\mu+1}}{n}\right)^\nu n^{q+p-1},
\eeqsn
for some $c,C',C'_\nu>0$,
and so
\beqs
\label{I3vk}
\langle I_3v_k,v_k\rangle\geq -C'\frac{n^{p-1}}{\varrho_k}\|v_k\|^2
-C'_\nu{\varrho_k}^{3+\frac12}\left(\frac{{\varrho_k}^{\mu+1}}{n}\right)^\nu 
n^{q+p-1}\|v_k\|.
\eeqs
Summing up \eqref{I1vk}, \eqref{I2vk} and \eqref{I3vk} 
we finally get the desired estimate \eqref{AB8}.
\end{proof}

\begin{Prop}
\label{lemma24}
Let $n=\varrho_k^{a}$ with $a>\mu+1$. Then for all $\nu\in\N$ 
there exists $C_\nu>0$ such that
the function $f_k$
defined in \eqref{fk} satisfies 
\beqsn
\|f_k(t,\cdot)\|\leq&& C\varrho_k^2n^{p-2}\sum_{j=1}^p\|v_k^{j,0}\|+C_\nu
n^{p-1}\!\!
\sum_{1\leq\alpha_1+\alpha_2\leq \nu(p-1)}\!\left(\frac{\varrho_k^\mu}{n}
\right)^{\alpha_1+\alpha_2}\|v_k^{\alpha_1,\alpha_2}\|\\
&&+C_\nu\,n^{q+p-1}\varrho_k^{4+\frac12}
\left(\frac{\varrho_k^{\mu+1}}{n}\right)^{\nu}
\eeqsn
for some fixed $C>0$ and for every $t\in[0,\varrho_k/n^{p-1}]$ with
$k$ large enough.
\end{Prop}

\begin{proof}
Let us recall that
\beqs
\label{addendi}
f_k=&&\op\left(a_p(t)\sum_{\alpha=2}^p
\binom p\alpha \xi^{p-\alpha}D_x^\alpha w_{n,k}\right)u_k
+\sum_{j=0}^{p-1}[a_j(t,x)D_x^j,W_{n,k}]u_k,
\eeqs
and estimate the above terms  separately.
For $\alpha=p$
\beqsn
\op(a_p(t)D_x^pw_{n,k})u_k=&&\int e^{ix\xi}a_p(t)D_x^pw_{n,k}(t,x;\xi)
\hat{u}_k(t,\xi)\dbar\xi\\
=&&a_p(t)\varrho_k^p\int e^{ix\xi}w_{n,k}^{p,0}(t,x;\xi)\hat{u}_k(t,\xi)\dbar\xi\\
=&&a_p(t)\varrho_k^pW_{n,k}^{p,0}(t,x;D_x)u_k(t,x)
=a_p(t)\varrho_k^pv_k^{p,0}(t,x)
\eeqsn
and hence
\beqs
\label{A9}
\|\op(a_p(t)D_x^pw_{n,k})u_k\|\leq C\varrho_k^p\|v_k^{p,0}\|
\eeqs
for some $C>0$.

For $2\leq\alpha\leq p-1$, by \eqref{2.28} we have:
\beqs
\nonumber
\|\op(a_p(t)\xi^{p-\alpha} D_x^\alpha w_{n,k})u_k(t,\cdot)\|\leq
&&C\varrho_k^\alpha\|W_{n,k}^{\alpha,0}D_x^{p-\alpha}u_k\|\\
\nonumber
\leq&&C'\varrho_k^\alpha\left(n^{p-\alpha}\sum_{j=0}^{p-\alpha}\|v_k^{\alpha+j,0}\|
+C_\nu\varrho_k^{4+\frac12}
\left(\frac{\varrho_k^{\mu+1}}{n}\right)^\nu n^{q+p-\alpha}\right)\\
\label{A1}
\leq&& C''\varrho_k^2n^{p-2}\sum_{s=2}^p\|v_k^{s,0}\|
+C'_\nu\varrho_k^{4+\frac12}
\left(\frac{\varrho_k^{\mu+1}}{n}\right)^\nu n^{q+p-1}
\eeqs
for some $C,C',C'',C_\nu,C'_\nu>0$, since 
$(\varrho_k/n)^\alpha\leq(\varrho_k/n)^2$ and 
$\varrho_k^2/n^2\leq1/n=\varrho_k^{-a}$ for
$2\leq\alpha\leq p-1$ and $a\geq2$.

In order to estimate the second addend of \eqref{addendi} we compute,
for $0\leq j\leq p-1$:
\beqsn
[a_jD_x^j,W_{n,k}]u_k
=&&a_j\sum_{h=0}^j\binom jh(D_x^{j-h}W_{n,k})D_x^hu_k
-W_{n,k}a_jD_x^ju_k\\
=&&a_j\sum_{h=0}^{j-1}\binom jh\varrho_k^{j-h}W_{n,k}^{j-h,0}D_x^hu_k
+[a_j,W_{n,k}]D_x^ju_k.
\eeqsn
Then, by Lemmas \ref{lemmaAB2} and \ref{lemmaAB4}, for $0\leq j\leq p-1$,
we have that:
\beqs
\nonumber
\|[a_jD_x^j,W_{n,k}]u_k\|\leq&&
C\sum_{h=0}^{j-1}\varrho_k^{j-h}\|W_{n,k}^{j-h,0}D_x^hu_k\|
+\|[a_j,W_{n,k}]D_x^ju_k\|\\
\nonumber
\leq&&\sum_{h=0}^{j-1}C_h\varrho_k^{j-h}n^h\sum_{s=0}^h\|v_k^{j-h+s,0}\|
+C_\nu \varrho_k^{4+\frac12}
\left(\frac{\varrho_k^{\mu+1}}{n}\right)^\nu n^{q+j}
\\
\nonumber
&&+C_\nu n^j\sum_{1\leq\alpha_1+\alpha_2\leq(\nu-1)(p-1)+j}
\left(\frac{\varrho_k^\mu}{n}\right)^{\alpha_1+\alpha_2}
\|v_k^{\alpha_1,\alpha_2}\|
\\
\label{A3}
\leq&&C\varrho_k n^{j-1}\sum_{s=1}^j\|v_k^{s,0}\|
+C_\nu \varrho_k^{4+\frac12}
\left(\frac{\varrho_k^{\mu+1}}{n}\right)^\nu n^{q+j}\\
\nonumber
&&+C_\nu n^j
\sum_{1\leq\alpha_1+\alpha_2\leq(\nu-1)(p-1)+j}
\left(\frac{\varrho_k^\mu}{n}\right)^{\alpha_1+\alpha_2}
\|v_k^{\alpha_1,\alpha_2}\|
\eeqs
for some $C,C_\nu>0$.

By \eqref{addendi}, \eqref{A9}, \eqref{A1} and \eqref{A3}:
\beqsn
\|f_k(t,\cdot)\|\leq&&
C\varrho_k^p\|v_k^{p,0}\|+C''\varrho_k^2n^{p-2}\sum_{s=2}^p\|v_k^{s,0}\|
+C'_\nu\varrho_k^{4+\frac12}
\left(\frac{\varrho_k^{\mu+1}}{n}\right)^\nu n^{q+p-1}\\
&&+\sum_{j=0}^{p-1}\bigg[C\varrho_kn^{j-1}\sum_{s=1}^j\|v_k^{s,0}\|
+C'_\nu n^j
\sum_{1\leq\alpha_1+\alpha_2\leq(\nu-1)(p-1)+j}
\left(\frac{\varrho_k^\mu}{n}\right)^{\alpha_1+\alpha_2}
\|v_k^{\alpha_1,\alpha_2}\|\\
&&+C'_\nu\varrho_k^{4+\frac12}
\left(\frac{\varrho_k^{\mu+1}}{n}\right)^\nu n^{q+j}\bigg]\\
\leq&&\tilde{C}\varrho_k^2n^{p-2}\sum_{s=1}^p\|v_k^{s,0}\|
+\tilde{C}_\nu n^{p-1}\sum_{1\leq\alpha_1+\alpha_2\leq\nu(p-1)}
\left(\frac{\varrho_k^\mu}{n}\right)^{\alpha_1+\alpha_2}
\|v_k^{\alpha_1,\alpha_2}\|\\
&&+\tilde{C}_\nu\varrho_k^{4+\frac12}
\left(\frac{\varrho_k^{\mu+1}}{n}\right)^\nu n^{q+p-1}
\eeqsn
for some $\tilde{C},\tilde{C}_\nu>0$.
\end{proof}

Summing up, from \eqref{encontrary2}, by Propositions \ref{ima2} and
\ref{lemma24}, for every $\nu\in\N$ we come to the estimate:
\beqs
\label{lemma3.1quaderno}
\nonumber
\frac12\frac{d}{dt}\|v_k(t,\cdot)\|^2&\geq&
\left( \Im a_{p-1}(t,x_k+pA_p(t)n^{p-1})n^{p-1}-
A\left(1+\frac{n^{p-1}}{\varrho_k}\right)\right)\|v_k\|^2
\\
&&-C_\nu{\varrho_k}^{4+\frac12}\left(\frac{{\varrho_k}^{\mu+1}}{n}\right)^\nu 
n^{q+p-1}\|v_k\|
-C\varrho_k^2n^{p-2}\sum_{j=1}^p\|v_k^{j,0}\|\cdot\|v_k\|\\
\nonumber
&&-C_\nu n^{p-1}
\sum_{1\leq\alpha_1+\alpha_2\leq \nu(p-1)}\!\left(\frac{\varrho_k^\mu}{n}
\right)^{\alpha_1+\alpha_2}
\|v_k^{\alpha_1,\alpha_2}\|\cdot\|v_k\|
\eeqs
for some $A,C,C_\nu>0$.
Now, for $a>\mu+1$, it is possible to take $\nu\in\N$ 
sufficiently large so that
\beqs
\label{sup}
\sup_k {\varrho_k}^{4+\frac12}\left(\frac{{\varrho_k}^{\mu+1}}{n}\right)^\nu 
n^{q+p-1}\leq M_\nu
\eeqs
for some $M_\nu>0$.
After substituting \eqref{sup} in \eqref{lemma3.1quaderno},
we finally choose $a$ and $\mu$ such that 
\beqs\label{postsup}
\frac{d}{dt}\|v_k(t,\cdot)\|&\geq&
\left( \Im a_{p-1}(t,x_k+pA_p(t)n^{p-1})n^{p-1}-
A\left(1+\frac{n^{p-1}}{\varrho_k}\right)\right)\|v_k\|- M'_\nu
\\
\nonumber
&&-C'_\nu n^{p-1}
\sum_{1\leq\alpha_1+\alpha_2\leq \nu(p-1)}\left(\frac{\varrho_k^\mu}{n}
\right)^{\alpha_1+\alpha_2}
\|v_k^{\alpha_1,\alpha_2}\|,
\eeqs
for some $M'_\nu,C'_\nu>0$; this can be done for
\beqs
\label{sceltamu}
\begin{cases}
\mu>p+1\cr
\ds\mu+1<a\leq\frac{p\mu-2}{p-1}=\mu+1+\frac{\mu-p-1}{p-1},
\end{cases}
\eeqs
since
$\varrho_k^2 n^{p-2}\leq n^{p-1}\left(\frac{\varrho_k^\mu}{n}
\right)^j$ for all $1\leq j\leq p$ if
$2\leq p\mu-a(p-1)$, and this implies, together with $a>\mu+1$, that 
we must take $\mu>p+1$.

Using now $\varrho_k\left(\frac{\varrho_k^\mu}{n}\right)^{\alpha_1+\alpha_2}
\leq \left(\frac{\varrho_k^{\mu+1}}{n}\right)^{\alpha_1+\alpha_2}$, we 
come to
\beqsn
\frac{d}{dt}\|v_k(t,\cdot)\|&\geq&\left( \Im a_{p-1}(t,x_k+pA_p(t)n^{p-1})n^{p-1}
-A\left(1+\frac{n^{p-1}}{\varrho_k}\right)\right)\|v_k\|- M'
\\
&&-C'\frac{n^{p-1}}{\varrho_k}
\sum_{1\leq\alpha_1+\alpha_2\leq \nu(p-1)}\!\left(\frac{\varrho_k^{\mu+1}}{n}
\right)^{\alpha_1+\alpha_2}
\|v_k^{\alpha_1,\alpha_2}\|
\eeqsn
for some constants $M',C'>0$, since $\nu$ has been fixed in \eqref{sup}.

Arguing in the same way for the functions $v_k^{\alpha,\beta}$ instead of $v_k$,
we finally get:
\begin{Prop}
\label{prop34}
Let $n$ be as in \eqref{n}, $a,\mu$ as in \eqref{sceltamu}, $\nu\in\N$
sufficiently large so that \eqref{sup} is satisfied. Then, for every
$\alpha, \beta\in\N_0$ there exists $C_{\alpha,\beta}>0$ such that
for all $t\in[0,\varrho_k/n^{p-1}]$ with $k$ large enough:
\beqs
\label{belowderivatives}
\frac{d}{dt}\|v_k^{\alpha,\beta}(t,\cdot)\|&\geq&
\left( \Im a_{p-1}(t,x_k+pA_p(t)n^{p-1})n^{p-1}-
A\left(1+\frac{n^{p-1}}{\varrho_k}
\right)\right)\|v_k^{\alpha,\beta}\|- C_{\alpha,\beta}
\\
\nonumber
&&-C_{\alpha,\beta}\frac{n^{p-1}}{\varrho_k}
\sum_{1\leq\tilde\alpha+\tilde \beta\leq \nu(p-1)}\!\left(\frac{\varrho_k^{\mu+1}}{n}
\right)^{\tilde\alpha+\tilde\beta}
\|v_k^{\alpha+\tilde\alpha,\beta+\tilde\beta}\|.
\eeqs
\end{Prop}

From Proposition \ref{prop34} it follows that:
\beqs
\label{bohh}
\nonumber
\frac{d}{dt}\left(\left(\frac{\varrho_k^{\mu+1}}n\right)^{\alpha+\beta}\!\!
\|v_k^{\alpha,\beta}\|\right)&\geq&\bigg( \Im a_{p-1}(t,x_k+pA_p(t)n^{p-1})n^{p-1}
-A\left(1+\frac{n^{p-1}}{\varrho_k}\right)\bigg)\\
&&\cdot
\left(\frac{\varrho_k^{\mu+1}}n
\right)^{\alpha+\beta}\!\!\|v_k^{\alpha,\beta}\|-C_{\alpha,\beta}
\\
\nonumber
&&-C_{\alpha,\beta}\frac{n^{p-1}}{\varrho_k}
\sum_{1\leq\tilde\alpha+\tilde \beta\leq \nu(p-1)}\!\left(\frac{\varrho_k^{\mu+1}}{n}
\right)^{\alpha+\tilde\alpha+\beta+\tilde\beta}
\|v_k^{\alpha+\tilde\alpha,\beta+\tilde\beta}\|.
\eeqs
We now choose $s\in\N$ sufficiently large so that, for all 
$\bar{\alpha}+\bar{\beta}\geq s+1$, using \eqref{normevkab}
and $a>\mu+1$, we have
\beqs
\label{costanti}
\frac{n^{p-1}}{\varrho_k}\left(\frac{\varrho_k^{\mu+1}}{n}
\right)^{\bar\alpha+\bar\beta}
\|v_k^{\bar\alpha,\bar\beta}\|\leq 
c_s\frac{n^{p-1}}{\varrho_k}\left(\frac{\varrho_k^{\mu+1}}{n}
\right)^{s+1}\varrho_k^{\frac 12+2}n^q\leq c'_s
\eeqs
for some $c_s,c'_s>0$. In order to satisfy \eqref{costanti}
it's enough to take $s$ such that
\beqsn
a(q+p-1)+\frac 12+1+(s+1)(\mu+1-a)\leq 0,
\eeqsn
i.e.
\beqs\label{chooses}
s\geq\frac{a(q+p-2)+\mu+\frac 52}{a-\mu-1}.
\eeqs
With this choice of $s$ we define:
\beqs\label{sigmak}
\sigma_k(t):=\ds\sum_{0\leq\alpha+\beta\leq s}\left(\frac{\varrho_k^{\mu+1}}{n}
\right)^{\alpha+\beta}
\|v_k^{\alpha,\beta}\|.
\eeqs
From \eqref{bohh} we have that: 
\beqsn
\frac{d}{dt}\sigma_k(t)&=&
\ds\sum_{0\leq\alpha+\beta\leq s}\frac{d}{dt}
\left[\left(\frac{\varrho_k^{\mu+1}}{n}
\right)^{\alpha+\beta}
\|v_k^{\alpha,\beta}\|\right]
\\
&\geq &\ds\sum_{0\leq\alpha+\beta\leq s}
\left( \Im a_{p-1}(t,x_k+pA_p(t)n^{p-1})n^{p-1}
-A\left(1+\frac{n^{p-1}}{\varrho_k}\right)\right)
\left(\frac{\varrho_k^{\mu+1}}n
\right)^{\alpha+\beta}\|v_k^{\alpha,\beta}\| \\
&&
-C_s\sum_{1\leq\bar\alpha+\bar\beta\leq s}\frac{n^{p-1}}{\varrho_k}
\left(\frac{\varrho_k^{\mu+1}}n\right)^{\bar\alpha+\bar\beta}
\|v_k^{\bar\alpha,\bar\beta}\|-C_s\\
&\geq &\ds\left( \Im a_{p-1}(t,x_k+pA_p(t)n^{p-1})n^{p-1}
-A_s\left(1+\frac{n^{p-1}}{\varrho_k}\right)\right)
\sigma_k(t)-C_s
\eeqsn
for some $C_s,A_s>0$, because of \eqref{costanti}. 

We have thus obtained for the function $\sigma_k$ the following 
differential inequality:
\beqsn
\sigma_k'(t)-B_k(t)\sigma_k(t)+C_s\geq 0\qquad 
t\in\left[0,\ds\frac{\varrho_k}{n^{p-1}}\right],\ k\gg1,
\\
\nonumber
B_k(t):=\Im a_{p-1}(t,x_k+pA_p(t)n^{p-1})n^{p-1}-A_s
\left(1+\frac{n^{p-1}}{\varrho_k}\right),
\eeqsn
which clearly implies that
\beqsn
\sigma_k(t)\geq e^{\int_0^t B_k(\theta)d\theta}\left[\sigma_k(0)-C_s\ds
\int_0^t e^{-\int_0^\tau B_k(\theta)d\theta}d\tau\right]
\qquad 
t\in\left[0,\ds\frac{\varrho_k}{n^{p-1}}\right], \ k\gg1.
\eeqsn
For $t=\varrho_k/n^{p-1}$ we have
\beqs\label{frombelow!!!}
\sigma_k\left(\frac{\varrho_k}{n^{p-1}}\right)\geq e^{\int_0^{\frac{\varrho_k}{n^{p-1}}} 
B_k(\theta)d\theta}\left[\sigma_k(0)-C_s\ds\int_0^{\frac{\varrho_k}{n^{p-1}}} 
e^{-\int_0^\tau B_k(\theta)d\theta}d\tau\right].
\eeqs
Let us focus on the term $\int_0^{\frac{\varrho_k}{n^{p-1}}} B_k(\theta)d\theta$; 
the choice of $x_k,\varrho_k$ of Lemma \ref{lemma1} gives for it, by
the change of variables $\theta'=n^{p-1}\theta$ and for $k$ large enough, the 
following estimate from below:

\beqs
\label{okbelow}
\nonumber
\int_0^{\frac{\varrho_k}{n^{p-1}}} B_k(\theta)d\theta&=&\int_0^{\frac{\varrho_k}{n^{p-1}}}
\Im a_{p-1}(\theta,x_k+pA_p(\theta)n^{p-1})n^{p-1}d\theta-
A_s\int_0^{\frac{\varrho_k}{n^{p-1}}}
\left(1+\frac{n^{p-1}}{\varrho_k}\right)d\theta
\\
\nonumber
&\geq &\int_0^{\varrho_k}\Im a_{p-1}\left(\frac{\theta'}{n^{p-1}},
x_k+pA_p\left(\frac{\theta'}{n^{p-1}}
\right)
n^{p-1}\right)d\theta'-2A_s
\\
\nonumber
 &= & \int_0^{\varrho_k}\Im a_{p-1}\left(\frac{\theta'}{n^{p-1}},
x_k+pa_p(\tau_k)\theta'\right)d\theta'-2A_s
\\
&\geq &M\log(1+\varrho_k)+k-2A_s,
\eeqs
for some $\tau_k\in[0,\theta'/n^{p-1}]$, since
$A_p(\theta'/n^{p-1})n^{p-1}=\theta'a_p(\tau_k)$
by the mean value theorem for integration.

Similarly it follows that for every $\tau\in [0,\frac{\varrho_k}{n^{p-1}}]$:
\beqs\label{okpos}
\int_0^{\tau} B_k(\theta)d\theta&\geq &
\int_0^{n^{p-1}\tau}
\Im a_{p-1}\left(\frac{\theta'}{n^{p-1}},
x_k+pa_p(\tau'_k)\theta'\right)d\theta'-2A_s\geq -2A_s
\eeqs
for some $\tau_k'\in[0,\theta'/n^{p-1}]$,
because of Lemma \ref{lemma1}, 
since $n^{p-1}\tau\leq n^{p-1}\frac{\varrho_k}{n^{p-1}}\leq
\varrho_k$.

Finally, from \eqref{sigmak} and \eqref{vk0} we have 
$\|\sigma_k(0)\|\geq \|v_k(0)\|\geq \|h\|>0$; therefore,
substituiting the estimates \eqref{okbelow} and \eqref{okpos} into 
\eqref{frombelow!!!}, we have proved the following 
desired estimate from below for the function $\sigma_k(t)$:
\begin{Prop}
\label{prop35}
For every $M>0$ and $k\in\N$ let $x_k,\varrho_k$  be as in Lemma~\ref{lemma1}.
Taking $\mu\geq 2$ in \eqref{eqw} and $n$ as in \eqref{n} with $a,\mu$
satisfying \eqref{sceltamu}, 
it is possible to construct the functions $v_k^{\alpha,\beta}$ in
\eqref{vkab} and then to choose $s$ 
great enough (see \eqref{chooses}) such that the function $\sigma_k(t)$ 
defined in \eqref{sigmak} satisfies the following estimate from below:
\beqs\label{formulabelow}
\sigma_k\left(\frac{\varrho_k}{n^{p-1}}\right)\geq c (1+\varrho_k)^{M},
\qquad k\gg1,
\eeqs
for some $c>0$.
\end{Prop}


\section{Estimate from above and proof of the main Theorem. }
\label{sec4}

The estimate from above is now quite simple to be obtained and it is 
shown in the following:
\begin{Prop}
\label{prop41}
For every $M>0$ and $k\in\N$ let $x_k,\varrho_k$  be as in Lemma~\ref{lemma1}.
Taking $\mu\geq 2$ in \eqref{eqw} and $n$ as in \eqref{n} with $a,\mu$
satisfying \eqref{sceltamu}, 
it is possible to construct the functions $v_k^{\alpha,\beta}$ in
\eqref{vkab} and then to choose $s$ 
great enough (see \eqref{chooses}) such that the function $\sigma_k(t)$ 
defined in \eqref{sigmak} satisfies the following estimate from above 
for all $t\in [0,\frac{\varrho_k}{n^{p-1}}]$:
\beqs\label{formulaabove}
\sigma_k(t)\leq C \varrho_k^{\frac 12+2+aq},\qquad k\gg1,
\eeqs    
for some $C>0$.
\end{Prop}
\begin{proof}
The estimate \eqref{normevkab} obtained in Section \ref{sec2} 
and definition \eqref{sigmak} immediately give:
\beqsn
\sigma_k(t)\leq \ds\sum_{0\leq\alpha+\beta\leq s}C_{\alpha,\beta}
\left(\frac{\varrho_k^{\mu+1}}{n}
\right)^{\alpha+\beta}\varrho_k^{\frac 12+2+aq}\leq C\varrho_k^{\frac 12+2+aq}
\eeqsn
for some $C>0$, since $s$ has been fixed in \eqref{chooses}.
\end{proof}

We are now ready to give the proof of Theorem \ref{th1}.
\begin{proof}[Proof of Theorem \ref{th1}]
Let us assume, by contradiction, that the Cauchy problem \eqref{CP1} is
well-posed in $H^\infty$ but \eqref{CN2} does not hold true.
Then at least one of the two conditions \eqref{A} or \eqref{B}
does not hold true.
As we remarked in Section~\ref{sec2}, we can assume, 
without loss of generality, that \eqref{A} does not hold and apply
Lemma \ref{lemma1}. By Propositions \ref{prop35} and \ref{prop41} we 
come to the estimate:
\[c (1+\varrho_k)^{M}\leq \sigma_k\left(\frac{\varrho_k}{n^{p-1}}\right)
\leq C\varrho_k^{\frac 12+2+aq},\]
for positive constants $c,C$ not depending on $k$,
giving rise to a contradiction for $k$ large enough, if we choose
\beqsn
M>\frac 12+2+aq.
\eeqsn

Therefore condition \eqref{CN2} must be satisfied and the proof is complete.
\end{proof}

\appendix
\section{}
The localized pseudo-differential operators $W_{n,k}^{\alpha,\beta}(t,x,D_x)$ 
of the present paper have symbols $w_{n,k}^{\alpha,\beta}(t,x,\xi)$ depending 
on the parameter $t$ and belonging to the class $S^{0}_{0,0}$ 
of all functions $p(x,\xi)\in C^\infty(\R^2)$ such that for every 
$\alpha,\beta\geq 0$ 
\beqs
\label{classe000}
\vert D_x^\beta\partial_\xi^\alpha p(x,\xi)\vert\leq C_{\alpha,\beta};
\eeqs
$S^{0}_{0,0}$ is a Fr\'echet space with semi-norms
\beqs
\label{semi000}
|p|_{\ell,\ell'}^{0}:=\max_{\alpha\leq\ell,\beta\leq \ell'}\sup_{x,\xi\in\R}\vert 
\partial_\xi^\alpha D_x^\beta p(x,\xi)\vert.
\eeqs
The class $S^0_{0,0}$ corresponds to the classical class $S^m_{\varrho,\delta}$ 
(defined by $\vert D_x^\beta\partial_\xi^\alpha p(x,\xi)\vert
\leq C_{\alpha,\beta}\langle\xi\rangle^{m-\varrho\alpha+\delta\beta}$ 
instead of \eqref{classe000}; see \cite{KG})
with 
$m=\varrho=\delta=0$. In the 
$S^m_{0,0}$ classes
the usual asymptotic expansion formula 
\[p(x,\xi)\sim\ds\sum_{\alpha\geq 0}\frac1{\alpha!}\partial_\xi^\alpha 
p_1(x,\xi)D_x^\beta p_2(x,\xi)\]
fails to be true, and we need to use the expansion formula with a 
remainder, as in \cite[Thm. 3.1, Chap. 2]{KG} (see also \cite[Thm. A]{I1}):
\begin{Th}
\label{thA}
Let $P_j(x,D_x)$ be pseudo-differential operators with symbols 
$p_j(x,\xi)\in S^{m_j}_{0,0}$, $j=1,2$.
Then the operator $P(x,D_x)=P_1(x,D_x)\circ P_2(x,D_x)$ has symbol 
given by the oscillatory integral
$$p(x,\xi)=\ds\int\!\!\int e^{-iy\eta}p_1(x,\xi+\eta)
p_2(x+y,\xi)dy\dbar\eta\in S^{m_1+m_2}_{0,0},$$
where $\dbar\eta=(2\pi)^{-1}d\eta$.

Moreover, the following expansion formula holds for every $\nu\in\N$:
\beqsn
p(x,\xi)=\ds\sum_{\alpha\leq\nu-1}\frac1{\alpha!}
\partial_\xi^\alpha p_1(x,\xi)D_x^\beta p_2(x,\xi)+\ds
\int_0^1\frac{(1-\theta)^{\nu-1}}{(\nu-1)!}r_{\theta,\nu}(x,\xi)d\theta,
\eeqsn
where
\beqsn
r_{\theta,\nu}(x,\xi):=\ds\int\!\!\int e^{-iy\eta}\partial_\xi^\nu 
p_1(x,\xi+\theta\eta)D_x^\nu p_2(x+y,\xi)dy\dbar\eta\in S^{m_1+m_2}_{0,0}.
\eeqsn
\end{Th}

We recall from \cite[Lemm 2.2, Chap. 7]{KG}, (see also \cite[Thm. B]{I1}):
\begin{Th}
\label{thB}
Let $p_j(x,\xi)\in S^0_{0,0}$ for $j=1,2$ and define
\beqsn
p_\theta(x,\xi):=\ds\int\!\!\int e^{-iy\eta}p_1(x,\xi+\theta\eta)
p_2(x+y,\xi)dy\dbar\eta.
\eeqsn
Then for every $\ell\in\N_0$ there exists a constant $C_\ell>0$ 
such that
\beqsn
|p_\theta|_{\ell,\ell}^{0}\leq C_\ell|p_1|_{\ell+2,\ell+2}^{0}
|p_2|_{\ell+2,\ell+2}^{0}
\eeqsn
for all $\theta\in[0,1]$.
\end{Th}

We conclude the appendix with the statement of the
Calder\'on-Vaillancourt's Theorem about continuity of 
pseudo-differential operators with symbols in the class $S^0_{0,0}$ 
acting on  $L^2$ (see \cite{CV} or \cite[Thm. C]{I2}):
\begin{Th}
\label{thC}
Let $p(x,D_x)$ be a pseudo-differential operator with symbol 
$p(x,\xi)\in S^0_{0,0}$. Then:
\beqsn
\|p(x,D_x)u\|\leq C \vert p\vert^0_{2,2}\ \|u\|
\eeqsn
for all $u\in L^2$, with a positive constant $C$ independent of $p$ and $u$. 
\end{Th}


{\bf Aknowledgment.}
The first two authors were partially supported by the INdAM-GNAMPA
Project 2015 ``Equazioni Differenziali a Derivate Parziali di
Evoluzione e Stocastiche''.

\end{document}